\documentclass[preprint,12pt,sort&compress]{elsarticle}

\usepackage{amssymb}
\usepackage{graphicx}
\usepackage{setspace}
\usepackage{amsthm,amsmath,amssymb,multirow,mathrsfs}
\usepackage[linkcolor=lue, urlcolor=blue, citecolor=blue, colorlinks,bookmarks]{hyperref}
\newtheorem{theorem}{Theorem}[section]
\newtheorem{Lemma}[theorem]{Lemma}

\newtheorem{example}[theorem]{Example}

\numberwithin{mytheorem}{subsection}

\numberwithin{mythe}{subsection}
\usepackage{lineno}
\usepackage{color}
\usepackage{ dsfont }

\journal{.}
\begin{document}

\begin{frontmatter}
\title{An efficient numerical method based on exponential B-splines for time-fractional Black-Scholes equation governing European options}
\author[label1]{Anshima Singh}
\address[label1]{Department of Mathematical Sciences, Indian Institute of Technology (BHU) Varanasi, Uttar Pradesh,
India}
\ead{anshima.singh.rs.mat18@iitbhu.ac.in}
\author[label1]{Sunil Kumar}


\begin{abstract}
In this paper a time-fractional Black-Scholes model (TFBSM) is considered to study the price change of the underlying fractal transmission system. We develop and analyze a numerical method to solve the TFBSM governing European options. The numerical method combines the exponential B-spline collocation to discretize in space and a finite difference method to discretize in time. The method is shown to be unconditionally stable using von-Neumann analysis. Also, the method is proved to be convergent of order two in space and $2-\mu$ is time, where $\mu$ is order of the fractional derivative.  We implement the method on various numerical examples in order to illustrate the accuracy of the method, and validation of the theoretical findings. In addition, as an application, the method is used to price several different European options such as the European call option, European put option, and European double barrier knock-out call option.
\end{abstract}


\end{frontmatter}
\section{Introduction}\label{sec1:intro}
\noindent In the market of finance, investors need to minimize and control risks. By market risks, we mean the chances of the deficit because of those aspects that impact the inclusive performance of the markets. Investors can manage these types of risks by investing in significant instruments that remove the risk of price volatility. Such instruments are known as financial derivatives. The value of a financial derivative depends on the functioning of the underlying asset. Among many financial derivatives, an option is one of the most well-known and important derivatives, so pricing an option is a significant problem both in practice and in theory. In 1973, Black and Scholes \cite{black2019pricing} and Merton \cite{merton1973theory} had proposed the Black\--Scholes model, which gives an accurate delineation of the behavior of the underlying asset. Black\--Scholes model is a second-order parabolic partial differential equation with respect to stock price and time, that governs the European option value on a stock, whose price pursues the geometric Brownian motion with fixed interest rate and constant volatality. Since the Black\--Scholes model is  simple and effective to model option value,  it led to a revolution in the financial market over the past several decades. The classical Black\--Scholes model is given  by \cite{rodrigo2006alternative, bohner2009analytical}
	\begin{eqnarray}\label{eq:1}
\frac{\partial  \mathscr{V}(\xi,\tau)}{\partial \tau}+\frac{\sigma^{2}\xi^{2}}{2}\frac{\partial^{2} \mathscr{V}(\xi,\tau)}{\partial \xi^2}+(r_{f}-D_{Y})\xi\frac{\partial  \mathscr{V}(\xi,\tau)}{\partial \xi}-r_{f} \mathscr{V}(\xi,\tau)=0, \hspace{0.10cm} (\xi,\tau)\in \mathds{R}^{+} \times (0,\widetilde{T}),
	\end{eqnarray}
	with the terminal condition
	\begin{eqnarray}\label{eq:2}
V(\xi,\widetilde{T})=\max(\xi_{\widetilde{T}}-\widetilde{K},0),\hspace{0.3cm} \xi_{\widetilde{T}} \geq 0,
	\end{eqnarray}
	where $ \mathscr{V}(\xi,\tau)$ denotes the value of a European option price. Here, $\sigma(\xi,\tau)$, $r_{f}(\tau)$, $D_{Y}(\tau)$, $\widetilde{T}$, $\widetilde{K}$,  and $\tau$ represent the volatility of the returns from the holding
	stock price $\xi$, the risk free rate, the dividend yield, the expiry time, the exercise price, and the current time respectively. Several approaches have been proposed in the literature to obtain a solution to the classical Black\--Scholes model \cite{amster2002solutions,amster2003stationary,
	company2008numerical,company2009numerical,
	bohner2009analytical,cen2011robust}. Due to the unrealistic assumptions used in the Black-Scholes model, it has some drawbacks, so it cannot explain a few existing phenomena such as stock price volatility, a short time boom in the financial market.  \cite{carr2003finite}, etc.
	
Fractional integrals and fractional derivatives are non-local, so they are a useful tool to desribe memory. Fractional differential equations have become a powerful tool for studying fractal dynamics and fractal geometry. The fractional calculus has had a massive impact on financial theory, as the fractional Black Scholes model can deal with most of the shortcomings of the classical Black Scholes model. Wyss was the first researcher to introduce the time-fractional Black Scholes model for pricing a European call option \cite{wyss2000fractional}. To deal with the problem of short time jumps in the financial market, Cartea and del-Castillo-Negrete \cite{cartea2007fractional} developed the space fractional Black-Scholes model to price exotic options. Jumarie \cite{jumarie2008stock,jumarie2010derivation} applied It$\hat{o}$'s lemma and the fractional-order Taylor's series method to obtain the time-space fractional Black-Scholes model. He studied the dynamics of the stock exchange and also considered  Merton's optimal portfolio to provide new results. Liang et al. \cite{liang2010option} derived a bi-fractional Black\--Merton\--Scholes model. With growth of applications of fractional models in financial field, researchers have shown interest in solving them 
analytically \cite{chen2015analytically,prathumwan2020solution,
ghandehari2014european,edeki2017analytical,fall2019black} and numerically \cite{zhang2016numerical,de2017numerically,
song2013solution,koleva2017numerical,
golbabai2017new,hariharan2013efficient,
mesgarani2021impact,an2021space,akram2021efficient}.

 Spline functions are prevalent in mathematics, computer science, engineering, etc. \cite{GuptaKadalbajoo2016,kadalvg2010,KADALBAJOO2009439,KADALBAJOO2008271}.  The use of piecewise cubic polynomial spline interpolation often gives undesirable inflexion points. The exponential spline interpolation method is a generalization of cubic splines and avoids these inflexion points. Pruess \cite{pruess1979alternatives,pruess1976properties} showed that exponential splines can generate co-monotone and co-convex interpolants and provided the remedy to the inflexion points issue. Pruess \cite{pruess1979alternatives}, Boorm\cite{de1978practical}, and McCartin \cite{mccartin1991theory} studied the exponential splines in detail. McCartin also showed that the exponential splines accept a basis of B-splines. They are used in approximating the solutions of various classes of problems in differential equations \cite{rao2008exponential,
zhu2017exponential,asvcomputational,ravi2021unconditionally}.

 
 To the best of our knowledge, there is no result in the literature on the collocation method based on exponential B-spline functions for the time\--fractional Black\--Scholes model. So, we propose an effective collocation technique to solve the time\--fractional Black\--Scholes problem numerically. To achieve this we use exponential B-spline functions to discretize the space derivative and apply a finite difference method to discretize the Caputo fractional derivative. It is proved that the method is unconditionally stable by means of von Neumann analysis. Further, it is shown that the proposed method is convergent of order $O(h_{\varkappa}^2,h_{t}^{2-\mu})$, where $h_{t}$ and $h_{\varkappa}$ represent mesh spacing in the time
		and space directions, respectively. We perform several numerical experiments to validate the theoretical findings. We use the proposed method to price three European options governed by a TFBSM, namely the European call option, the European put option, and the European double barrier knock-out call option. In addition, we examine how the order of time-fractional derivative affects the option price.

The remaining paper is as structured follows. Section \ref{sec:tmp} describes the model problem.  Section \ref{sec:2} is devoted to the proposed numerical method. The stability and convergence of the method are discussed in Sections \ref{sec:3} and \ref{sec:4} respectively. To validate our theoretical findings, several numerical examples are considered in Section \ref{sec:5}. Finally, the paper is concluded in Section \ref{sec:6}.



\section{The model problem}\label{sec:tmp}
Consider the following time\--fractional Black\--Scholes model (TFBSM) \cite{chen2015analytically} that expresses the option price problem
	\begin{equation}\label{eq:3}
	\frac{\partial^{\mu} \mathscr{V}(\xi,\tau)}{\partial \tau^{\mu}} +\frac{\sigma^{2}\xi^{2}}{2}\frac{\partial^{2} \mathscr{V}(\xi,\tau)}{\partial \xi^2}+r_{f}\xi\frac{\partial  \mathscr{V}(\xi,\tau)}{\partial \xi}-r_{f}  \mathscr{V}(\xi,\tau)=0, \hspace{0.3cm} (\xi,\tau)\in \mathds{R}^{+} \times (0,\widetilde{T}),
		\end{equation}
	with
		\begin{equation}\label{eq:4}
	\left\{%
	\begin{array}{ll}
		\mathscr{V}(\xi,\widetilde{T})=\phi(\xi),\\	
	\mathscr{V}(0,\tau)=\widetilde{\mathscr{H}}(\tau),\\
		\mathscr{V}(\infty,\tau)=\widetilde{\mathscr{G}}(\tau),
	\end{array}%
	\right.~~~~~~~~~~~~~~~~~~~~~~~~~~~~~~~~~~~~~~~~~~~~~~~~~~~~~~~~~~~~~~~~~~~~~~~~~~~~~~~~~~
	\end{equation}
	where $0<\mu\leq 1$ and all other notations are same as defined for problem $(\ref{eq:1})$. Further, the modified Riemann\--Liouville fractional time derivative is defined by \cite{podlubny1998fractional} 
		\begin{equation}\label{eq:5}
	\displaystyle \frac{\partial^{\mu} \mathscr{V}(\xi,\tau)}{\partial \tau^{\mu}}	= \displaystyle\begin{cases} 
	\frac{1}{\Gamma(1-\mu)}\frac{d}{d\tau}\int_{\tau}^{\widetilde{T}}\frac{\mathscr{V}(\xi,\nu)-\mathscr{V}(\xi,\widetilde{T})}{(\nu-\tau)^{\mu}}d\nu, & 0<\mu<1, \\  
\frac{\partial \mathscr{V}}{\partial \tau} & \mu=1.    
	\end{cases}
	\end{equation}
When $\mu=1,$ the model ($\ref{eq:3}$)\--($\ref{eq:4}$) converts to the classical Black\--Scholes model ($\ref{eq:1}$)\--($\ref{eq:2}$).

We consider the transformation $\tau=\widetilde{T}-t$, for $0<\mu<1,$ and proceed as follows 
\begin{eqnarray*}
\frac{\partial^{\mu} \mathscr{V}(\xi,\tau)}{\partial \tau^{\mu}}&=&	\frac{1}{\Gamma(1-\mu)}\frac{d}{d\tau}\int_{\tau}^{\widetilde{T}}\frac{\mathscr{V}(\xi,\nu)-\mathscr{V}(\xi,\widetilde{T})}{(\nu-\tau)^{\mu}}d\nu~~~~~~~~~~~~{}\\
&=&\frac{-1}{\Gamma(1-\mu)} \frac{d}{dt}\int_{\widetilde{T}-t}^{\widetilde{T}}\frac{\mathscr{V}(\xi,\nu)-\mathscr{V}(\xi,\widetilde{T})}{(\nu-\widetilde{T}+t)^{\mu}}d\nu\\
~~~~~~&=&\frac{-1}{\Gamma(1-\mu)} \frac{d}{dt}\int_{0}^{t}\frac{\mathscr{V}(\xi,\widetilde{T}-\tilde{\xi})-\mathscr{V}(\xi,\widetilde{T})}{(t-\tilde{\xi})^{\mu}}d\tilde{\xi}~~(\mbox{using}~ \nu=\widetilde{T}-\tilde{\xi})\\
&=&\frac{-1}{\Gamma(1-\mu)} \frac{d}{dt}\int_{0}^{t}\frac{u(\varkappa,\tilde{\xi})-u(\varkappa,0)}{(t-\tilde{\xi})^{\mu}}d\tilde{\xi}.\\
\end{eqnarray*}
Now defining
\begin{eqnarray*}
	{}_{0}^{}D^{\mu}_{t}u(\varkappa,t)&=&\frac{1}{\Gamma(1-\mu)} \frac{d}{dt}\int_{0}^{t}\frac{u(\varkappa,\tilde{\xi})-u(\varkappa,0)}{(t-\tilde{\xi})^{\mu}}d\tilde{\xi},~~ (0<\mu<1),
\end{eqnarray*}
we have
\[\frac{\partial^{\mu} \mathscr{V}(\xi,\tau)}{\partial \tau^{\mu}}=-{}_{0}^{}D^{\mu}_{t}u(\varkappa,t).\]
 Therefore, letting $\xi=e^x$ and denoting $u(\varkappa,t)=\mathscr{V}(e^\varkappa,\widetilde{T}-t),$ we rewrite $(\ref{eq:3})$\--($\ref{eq:4}$) as follows
	\begin{equation}\label{eq:6}
	{}_{0}^{}D^{\mu}_{t}u(\varkappa,t)= \frac{\sigma^{2}}{2}\frac{\partial^{2}u(\varkappa,t)}{\partial \varkappa^2}+(r_{f}-\frac{\sigma^{2}}{2})\frac{\partial u(\varkappa,t)}{\partial x}-r_{f} u(\varkappa,t), \hspace{0.3cm} (\varkappa,t)\in \mathds{R} \times (0,\widetilde{T}),
\end{equation}
with
\begin{equation}\label{eq:7}
\left\{%
\begin{array}{ll}
u(\varkappa,0)=z(\varkappa),\\	
u(-\infty,t)=\mathscr{H}(t),\\
u(\infty,t)=\mathscr{G}(t).
\end{array}%
\right.~~~~~~~~~~~~~~~~~~~~~~~~~~~~~~~~~~~~~~~~~~~~~~~~~~~~~~~~~~~~~~~~~~~~~~~~~~~~~~~~~~
\end{equation}
The modified Riemann\--Liouville fractional time derivative operator ${}_{0}^{}D^{\mu}_{t}$ can be transformed into the Caputo fractional time derivative operator ${}_{0}^{C}D^{\mu}_{t}$ for $0<\mu \leq 1$ following \cite{zhang2016numerical}:
\begin{eqnarray*}
	{}_{0}^{}D^{\mu}_{t}u(\varkappa,t)&=&\frac{1}{\Gamma(1-\mu)} \frac{d}{dt}\int_{0}^{t}\frac{u(\varkappa,\tilde{\xi})-u(\varkappa,0)}{(t-\tilde{\xi})^{\mu}}d\tilde{\xi},\\
	&=&\frac{1}{\Gamma(1-\mu)} \frac{d}{dt}\int_{0}^{t}\frac{u(\varkappa,\tilde{\xi})}{(t-\tilde{\xi})^{\mu}}d\tilde{\xi}-\frac{1}{\Gamma(1-\mu)} \frac{d}{dt}\int_{0}^{t}\frac{u(\varkappa,0)}{(t-\tilde{\xi})^{\mu}}d\tilde{\xi},\\
	&=&\frac{1}{\Gamma(1-\mu)} \frac{d}{dt}\int_{0}^{t}\frac{u(\varkappa,\tilde{\xi})}{(t-\tilde{\xi})^{\mu}}d\tilde{\xi}-\frac{t^{-\mu}}{\Gamma{(1-\mu)}}u(\varkappa,0),\\
		&=&\frac{1}{\Gamma(1-\mu)} \int_{0}^{t}(t-\tilde{\xi})^{-\mu}\frac{\partial u(\varkappa,\tilde{\xi})}{\partial \tilde{\xi}}d\tilde{\xi},\\
			&=&{}_{0}^{C}D^{\mu}_{t}u(\varkappa,t).
\end{eqnarray*}
Lastly, to solve the problem numerically, we need to truncate the unbounded domain into a finite interval $(I_{p},F_{p})$. Also, without loss of generality we add a source term to the RHS of the equation $(\ref{eq:6})$. Thus, we have the following problem
	\begin{equation}\label{eq:8}
{}_{0}^{C}D^{\mu}_{t}u(\varkappa,t)= \alpha\frac{\partial^{2}u(\varkappa,t)}{\partial \varkappa^2}+\beta\frac{\partial u(\varkappa,t)}{\partial x}-\gamma u(\varkappa,t)+\psi(\varkappa,t), \hspace{0.3cm} (\varkappa,t)\in (I_{p},F_{p}) \times (0,\widetilde{T}),
\end{equation}
with
\begin{equation}\label{eq:9}
\left\{%
\begin{array}{ll}
u(\varkappa,0)=z(\varkappa),\\	
u(I_{p},t)=\mathscr{H}(t),\\
u(F_{p},t)=\mathscr{G}(t),
\end{array}%
\right.~~~~~~~~~~~~~~~~~~~~~~~~~~~~~~~~~~~~~~~~~~~~~~~~~~~~~~~~~~~~~~~~~~~~~~~~~~~~~~~~~~
\end{equation}
where $\alpha=\frac{\sigma^{2}}{2}>0$, $\beta=r_{f}-\frac{\sigma^{2}}{2}$, and $\gamma=r_{f}>0$. From the above fractional Black\--Scholes model, we can obtain the widely known reaction\--diffusion model by taking $\alpha>0$, $\beta=0$, and $\gamma \neq 0$, and the time fractional advection\--diffusion model by taking $\alpha>0$, $\beta<0$, and $\gamma=0$. One find several works for both the models, but from the existing literature, it seems that the work on the time\--fractional Black\--Scholes model is comparatively less and confined. Therefore, in this paper, we have considered the time\--fractional Black\--Scholes model to solve it numerically using the collocation method based on the exponential B-spline functions.

\section{Numerical Scheme }\label{sec:2}
\subsection{\textbf{ \emph{Outline of the Exponential B\--spline functions}}}\label{sub:1}
  \noindent Let $\Pi_{\varkappa}: I_{p}=\varkappa_{0}<\varkappa_{1}<\dots<\varkappa_{N_{\varkappa}-1}<\varkappa_{N_{\varkappa}}=F_{p}$ be the uniform partition of $[I_{p},F_{p}]$ with $\varkappa_{m}=I_{p}+mh_{\varkappa}$, where $m=0,1,2,\dots,N_{\varkappa},$ and $h_{\varkappa}=\frac{(F_{p}-I_{p})}{N_{\varkappa}}$ is the mesh spacing in space direction. Suppose
  \begin{eqnarray*}\tilde{s}=\sinh(\rho h_{\varkappa}), \\ \tilde{c}=\cosh(\rho h_{\varkappa}),\end{eqnarray*}
  where $\rho$ is a non\--negative tension parameter.
   A generalization of the cubic spline as an exponential spline has been proposed by McCartin \cite{mccartin1991theory}.
The presence of non\--negative tension parameter plays an important role in the exponential spline and it accommodates the stiffness of multiple spline segments. The exponential B-spline functions are defined on the above mentioned partition $\Pi_{\varkappa}$ along with six more points $\varkappa_{i}$, where $i=-3,-2,-1,N_{\varkappa}+1,N_{\varkappa}+2,N_{\varkappa}+3$, which are outside of the interval $[I_{p},F_{p}]$. The exponential B-spline functions $\mathscr{Q}_{m}(\varkappa)$ are defined as follows
\begin{equation}\label{eq:10}
\displaystyle \mathscr{Q}_{m}(\varkappa) = \begin{cases} 
\tilde{r}(\varkappa_{m-2}-\varkappa)-\frac{\tilde{r}}{\rho}\sinh(\rho(\varkappa_{m-2}-\varkappa)), & \text{$\varkappa\in [\varkappa_{m-2},\varkappa_{m-1}]$,} \\  
\tilde{a}+\tilde{b}(\varkappa_{m}-\varkappa)+\bar{c}e^{\rho(\varkappa_{m}-\varkappa)}+qe^{-\rho(\varkappa_{m}-\varkappa)}, & \text{$\varkappa\in [\varkappa_{m-1},\varkappa_{m}]$,} \\  
\tilde{a}+\tilde{b}(\varkappa-\varkappa_{m})+\bar{c}e^{\rho(\varkappa-\varkappa_{m})}+qe^{-\rho(\varkappa-\varkappa_{m})}, & \text{$\varkappa\in [\varkappa_{m},\varkappa_{m+1}]$,}\\  
\tilde{r}(\varkappa-\varkappa_{m+2})-\frac{\tilde{r}}{\rho}\sinh(\rho(\varkappa-\varkappa_{m+2})), & \text{$\varkappa\in [\varkappa_{m+1},\varkappa_{m+2}]$,} \\   
0, & \emph{otherwise},  
\end{cases}
\end{equation}
where $m = -1, 0, \dots, N_{\varkappa}, N_{\varkappa} + 1$, and
\begin{eqnarray*}
\tilde{r}=\frac{\rho}{2(\rho h_{\varkappa}\tilde{c}-\tilde{s})},~~~~~ \tilde{a}=\frac{\rho h_{\varkappa}\tilde{c}}{\rho h_{\varkappa}\tilde{c}-\tilde{s}},~~~~~~\tilde{b}=\frac{\rho}{2}\left[\frac{\tilde{c}(\tilde{c}-1)+\tilde{s}^{2}}{(\rho h_{\varkappa}\tilde{c}-\tilde{s})(1-\tilde{c})}\right],\end{eqnarray*}
\begin{equation*}
\bar{c}=\frac{1}{4}\left[\frac{e^{-\rho h_{\varkappa}}(1-\tilde{c})+\tilde{s}(e^{-\rho h_{\varkappa}}-1)}{(\rho h_{\varkappa}\tilde{c}-\tilde{s})(1-\tilde{c})}\right], ~~~~~~~~~~~~~~~~~~~~~~~~~~~~~~~~~~~~~~~~~~~~~
\end{equation*}
\begin{equation*}
q=\frac{1}{4}\left[\frac{e^{\rho h_{\varkappa}}(\tilde{c}-1)+\tilde{s}(e^{\rho h_{\varkappa}}-1)}{(\rho h_{\varkappa}\tilde{c}-\tilde{s})(1-\tilde{c})}\right]. ~~~~~~~~~~~~~~~~~~~~~~~~~~~~~~~~~~~~~~~~~~~~~~~~~
\end{equation*}
The exponential B\--spline functions $\mathscr{Q}_{-1}(\varkappa)$, $\mathscr{Q}_{0}(\varkappa)$, $\dots$, $\mathscr{Q}_{N_{\varkappa}}(\varkappa)$, and $\mathscr{Q}_{N_{\varkappa}+1}(\varkappa)$ are twice continuously differentiable over $\mathds{R}$ and have local support. The set $\{\mathscr{Q}_{m}(\varkappa)\}_{m=-1}^{N_{\varkappa}+1}$ is linearly independent and forms a basis for the exponential B\--spline space $\mathscr{W}_{N_{\varkappa}+3}=\mbox{span}(\{\mathscr{Q}_{m}(\varkappa)\}_{m=-1}^{N_{\varkappa}+1})$ over the interval $[I_{p},F_{p}]$. The values of $\mathscr{Q}_{m}(\varkappa)$, $\mathscr{Q}_{m}'(\varkappa)$, and $\mathscr{Q}_{m}''(\varkappa)$ at each mesh point are given in the Table \ref{first}.
\begin{table}
	\caption{ The values of $\mathscr{Q}_{m}(\varkappa)$, $\mathscr{Q}_{m}'(\varkappa)$, and $\mathscr{Q}_{m}''(\varkappa)$ at each mesh point.}
~~~~~~~~~~~~~~~~\begin{tabular}{ |p{2cm} p{2.5cm} p{2.5cm} p{2.5cm}|}	
	\multicolumn{4}{c}{} \\
	\hline
	&$\mathscr{Q}_{m}(\varkappa_{i})$&$\mathscr{Q}_{m}'(\varkappa_{i})$&$\mathscr{Q}_{m}''(\varkappa_{i})$\\
	\hline
	$i=m$	&1&0&$\frac{-\rho^{2}\tilde{s}}{\rho h_{\varkappa}\tilde{c}-\tilde{s}}$\\
	$i=m\pm 1$	&$\frac{\tilde{s}-\rho h_{\varkappa}}{2(\rho h_{\varkappa}\tilde{c}-\tilde{s})}$&$\frac{\mp \rho(\tilde{c}-1)}{2(\rho h_{\varkappa}\tilde{c}-\tilde{s})}$&$\frac{\rho^{2}\tilde{s}}{2(\rho h_{\varkappa}\tilde{c}-\tilde{s})}$\\
	else	&0&0&0\\	
	\hline
\end{tabular}\label{first}
\end{table}
\subsection{\textbf{\emph{ A fully discrete numerical scheme}}}\label{sub:2}
 \noindent In this sub-section, first we will discretize the Caputo time-fractional derivative and then derive the fully-discretized numerical scheme. Let  $\Pi_{t}: 0=t_{0}<t_{1}<\cdots<t_{N_{t}-1}<t_{N_{t}}=\widetilde{T}$ be the uniform partition of $[0,\widetilde{T}]$, where  $t_{n}=n h_{t}$, $n=0,1,2,\cdots,N_{t},$ with $h_{t}=\frac{\widetilde{T}}{N_{t}}$ is the mesh spacing in the time direction. We approximate the Caputo time fractional derivative ${}_{~0}^{~C}D^{\mu}_{t}u(\varkappa,t)$ at $t_{n+1}$, $n=-1,0,\cdots,N_{t}-1,$ using $L1$ method \cite{li2019numerical} as follows
\begin{eqnarray*}
^{C}_{0}D^{\mu}_{t}u(\varkappa,t_{n+1})&=& 
\frac{1}{\Gamma{(1-\mu)}}\int_{0}^{t_{n+1}}(t_{n+1}-s)^{-\mu}\frac{\partial u}{\partial s}(\varkappa,s)ds\\
&=&\frac{1}{\Gamma{(1-\mu)}}\sum_{k=0}^{n}\int_{t_{k}}^{t_{k+1}}(t_{n+1}-s)^{-\mu}\left[\frac{u(\varkappa,t_{k+1})-u(\varkappa,t_{k})}{h_{t}}\right] ds + R^{n+1}
\end{eqnarray*}
\begin{equation}\label{eq:11}
~~=\frac{h_{t} ^{-\mu}}{\Gamma{(2-\mu)}}\sum_{k=0}^{n}w_{k} \left[u(\varkappa,t_{n-k+1})-u(\varkappa,t_{n-k})\right]+ R_{1}^{n+1},
\end{equation}
where $ w_{k}= (k+1)^{1-\mu}-(k)^{1-\mu}$ and the truncation error $R_{1}^{n+1}$ is bounded by
\begin{equation}\label{eq:tr}
|R^{n+1}|\leq k_{1} h_{t} ^{2-\mu},\end{equation}
where $k_{1}$ is a constant.\\

\begin{Lemma}\label{lemo}
	The coefficients  $ w_{k}$ satisfy \cite{mohyud2018fully}\\
		\noindent $(a)$.\hspace{0.2cm} $w_{0}=1$,\\
	$(b)$.\hspace{0.2cm} $w_{k}>0,$\hspace{0.3cm}  $0\leq k \leq n$,\\
	$(c)$.\hspace{0.2cm} $<w_{k}>$ is monotonic decreasing sequence,\\
	$(d)$.\hspace{0.2cm} $	\sum_{k=0}^{n}(w_{k}-w_{k+1})+w_{n+1}=1$.
	\end{Lemma}
At time level $n+1,$ the equation (\ref{eq:8}) takes the form
	\begin{equation}\label{eq7}
{}_{0}^{C}D^{\mu}_{t}u(\varkappa,t_{n+1})= \alpha\frac{\partial^{2}u^{n+1}(\varkappa)}{\partial \varkappa^2}+\beta\frac{\partial u^{n+1}(\varkappa)}{\partial x}-\gamma u^{n+1}(\varkappa)+\psi^{n+1}(\varkappa).
\end{equation}
Now we discretize equation (\ref{eq7}). Substitution of equation (\ref{eq:11}) in equation (\ref{eq7}) gives
\begin{eqnarray}\label{eq}
	\sum_{k=0}^{n}w_{k}\left [u(\varkappa,t_{n-k+1})-u(\varkappa,t_{n-k})\right]+R^{n+1} =\alpha \widetilde{\gamma}\frac{\partial^{2}u^{n+1}(\varkappa)}{\partial \varkappa^2}+\beta \widetilde{\gamma}\frac{\partial u^{n+1}(\varkappa)}{\partial \varkappa}- \gamma \widetilde{\gamma} u^{n+1}(\varkappa)+\widetilde{\gamma}\psi^{n+1}(\varkappa),
\end{eqnarray}
~~~~~~~~~~~~~~~~~~~~~~~~~~~~~~~~~~~~~~~~~~~~~~~~~~~~~~~~~~~~~~~~~~~~~~~~~~~~~~~~~~~~~~~~~~~~~~~~~$(-1\leq n \leq N_{t}-1),$\\
where
\begin{eqnarray*}
 	\widetilde{\gamma}=\Gamma{(2-\mu)} h_{t} ^{\mu}.
\end{eqnarray*}
The initial and boundary conditions in (\ref{eq:9}) result in
\begin{equation}\label{eq:13}
\left\{%
\begin{array}{ll}
u^0(\varkappa)=z(\varkappa),\\
u^{n+1}(I_{p})=\mathscr{H}(t_{n+1}), & -1\leq n \leq N_{t}-1,\\
u^{n+1}(F_{p})=\mathscr{G}(t_{n+1}), & -1\leq n \leq N_{t}-1.
\end{array}%
\right.
\end{equation}
\noindent The approximate solution $\mathscr{U}^{n+1}(\varkappa)$ to the analytical solution $u^{n+1}(\varkappa)$ of problem (\ref{eq:8})-(\ref{eq:9}) is considered to be in the following form
\begin{equation}\label{eq:14}
 \mathscr{U}^{n+1}(\varkappa)=\sum_{m=-1}^{N_{\varkappa}+1}\mathscr{R}^{n+1}_{m}\mathscr{Q}_{m}(\varkappa),
\end{equation}
	where $\mathscr{R}^{n+1}_{m}$ are unknown coefficients that need to be determined. With the help of Table \ref{first} we can obtain $\mathscr{U}^{n+1}(\varkappa_{m})$, $\mathscr{U}^{n+1}_{\varkappa}(\varkappa_{m})$, and $\mathscr{U}^{n+1}_{\varkappa\varkappa}(\varkappa_{m}),$ for $-1\leq m \leq N_{\varkappa}+1,$ in terms of the coefficients $\mathscr{R}_{m}^{n+1}$ as follows
\begin{eqnarray}
\mathscr{U}^{n+1}(\varkappa_{m})=\eta \mathscr{R}_{m-1}^{n+1}+\mathscr{R}_{m}^{n+1}+\eta \mathscr{R}_{m+1}^{n+1},\label{eq:15}\\
\mathscr{U}^{n+1}_{\varkappa}(\varkappa_{m})=\tilde{e}(\tilde{c}-1) [\mathscr{R}_{m+1}^{n+1}- \mathscr{R}_{m-1}^{n+1}],\label{eq:16}\\
\mathscr{U}^{n+1}_{\varkappa\varkappa}(\varkappa_{m})=\bar{\eta}[ \mathscr{R}_{m-1}^{n+1}-2\mathscr{R}_{m}^{n+1}+\mathscr{R}_{m+1}^{n+1}],\label{eq:17}
\end{eqnarray}	
where
\begin{eqnarray*}
	\eta=\frac{\tilde{s}-\rho h_{\varkappa}}{2(\rho  h_{\varkappa}\tilde{c}-\tilde{s})},~~~~~~ \tilde{e}=\frac{\rho}{2(\rho h_{\varkappa}\tilde{c}-\tilde{s})},~~~~~~ \bar{\eta}=\frac{\rho^{2}\tilde{s}}{2(\rho h_{\varkappa}\tilde{c}-\tilde{s})}.
\end{eqnarray*}	
Now the discretization of (\ref{eq}) at $\varkappa=\varkappa_{m}$ gives
\begin{eqnarray*}\label{eq1}
\sum_{k=0}^{n}w_{k}\left [u(\varkappa_{m},t_{n-k+1})-u(\varkappa_{m},t_{n-k})\right]+R^{n+1} =\alpha \widetilde{\gamma}\frac{\partial^{2}u^{n+1}(\varkappa_{m})}{\partial \varkappa_{m}^2}+\beta \widetilde{\gamma}\frac{\partial u^{n+1}(\varkappa_{m})}{\partial \varkappa_{m}}- \gamma \widetilde{\gamma} u^{n+1}(\varkappa_{m})\\+\widetilde{\gamma}\psi^{n+1}(\varkappa_{m}),
\end{eqnarray*}
\begin{equation}\label{eq2}
~~~~~~~~~~~~~~~~~~~~~~~~~~~~~~~~~~~~~~~~~~~~~~~~~~~~~~~~~~~~~~~~~~~~~~~(0\leq m \leq N_{\varkappa}, -1\leq n \leq N_{t}-1),
\end{equation}
The exponential B-spline function $\mathscr{U}^{n+1}(\varkappa)$ satisfies the collocation conditions as follows
\begin{eqnarray*}
	\sum_{k=0}^{n}w_{k}\left [\mathscr{U}^{n-k+1}(\varkappa_{m})-\mathscr{U}^{n-k}(\varkappa_{m})\right] =\alpha \widetilde{\gamma}\frac{\partial^{2}\mathscr{U}^{n+1}(\varkappa_{m})}{\partial \varkappa^2}+\beta \widetilde{\gamma}\frac{\partial \mathscr{U}^{n+1}(\varkappa_{m})}{\partial x}- \gamma \widetilde{\gamma} \mathscr{U}^{n+1}(\varkappa_{m})\\+\widetilde{\gamma}\psi^{n+1}(\varkappa_{m}),
\end{eqnarray*}
\begin{equation}\label{eq3}
~~~~~~~~~~~~~~~~~~~~~~~~~~~~~~~~~~~~~~~~~~~~~~~~~~~~~~~~~~~~~~~~~~~~~~~~(0\leq m \leq N_{\varkappa}, -1\leq n \leq N_{t}-1)
\end{equation}
\begin{equation}\label{eq4}
\left\{%
\begin{array}{ll}
\mathscr{U}^0(\varkappa_{m})=z_{m},\\
\mathscr{U}^{n+1}(\varkappa_{0})=\mathscr{H}_{n+1}, & -1\leq n \leq N_{t}-1,\\
\mathscr{U}^{n+1}(\varkappa_{N_{\varkappa}})=\mathscr{G}_{n+1}, & -1\leq n \leq N_{t}-1,
\end{array}
\right.
\end{equation}
 where $z(\varkappa_{m})=z_{m}$, $\mathscr{H}(t_{n+1})=\mathscr{H}_{n+1}$, and $\mathscr{G}(t_{n+1})=\mathscr{G}_{n+1}$.
  
 Now the substitution of the equations (\ref{eq:15}), (\ref{eq:16}), and (\ref{eq:17}) in equation (\ref{eq3}), and the substitution of equation (\ref{eq:15}) in boundary conditions given in equation (\ref{eq4}) yield

\begin{eqnarray*}
	\chi_{1}\mathscr{R}_{m-1}^{n+1}+\chi_{2}\mathscr{R}_{m}^{n+1}+\chi_{3}\mathscr{R}_{m+1}^{n+1}=\eta \mathscr{R}_{m-1}^{n}+\mathscr{R}_{m}^{n}+\eta \mathscr{R}_{m+1}^{n}- \sum_{k=1}^{n}w_{k}[(\eta\mathscr{R}_{m-1}^{n-k+1}+\mathscr{R}_{m}^{n-k+1}+\eta\mathscr{R}_{m+1}^{n-k+1})-\\(\eta\mathscr{R}_{m-1}^{n-k}+\mathscr{R}_{m}^{n-k}+\eta\mathscr{R}_{m+1}^{n-k})]+\widetilde{\gamma}\psi^{n+1}(\varkappa_{m}),~~~~~
\end{eqnarray*}
\begin{equation}\label{eq:18}
\end{equation}
and
	\begin{eqnarray}
\eta\mathscr{R}_{-1}^{n+1}=\mathscr{H}_{n+1}-\mathscr{R}_{0}^{n+1}-\eta\mathscr{R}_{1}^{n+1},\label{eq:19}\\
\eta\mathscr{R}_{N_{\varkappa}+1}^{n+1}=\mathscr{G}_{n+1}-\mathscr{R}_{N_{\varkappa}}^{n+1}-\eta\mathscr{R}_{N_{\varkappa}-1}^{n+1},\label{eq:20}
\end{eqnarray}	
where
\begin{eqnarray*}
	\chi_{1}&=&\eta-\alpha \widetilde{\gamma}\bar{\eta}+\beta \widetilde{\gamma}\tilde{e}(\tilde{c}-1)+\eta \gamma \widetilde{\gamma}, 
	\\ \chi_{2}&=&1+2\alpha \widetilde{\gamma}\bar{\eta}+\gamma \widetilde{\gamma},\\
	\chi_{3}&=&\eta-\alpha \widetilde{\gamma}\bar{\eta}-\beta \widetilde{\gamma}\tilde{e}(\tilde{c}-1)+\eta \gamma \widetilde{\gamma}.
\end{eqnarray*}
From the system (\ref{eq:18}) the unknown coefficients $\mathscr{R}_{-1}^{n+1}$ and $\mathscr{R}_{N_{\varkappa}+1}^{n+1}$ can be eliminated using the equations (\ref{eq:19}) and (\ref{eq:20}) respectively. Finally, we get for each $n$ a tri-diagonal system of $(N_{\varkappa} + 1)$ equations in $(N_{\varkappa} + 1)$ unknowns which can be solved by the Thomas algorithm. We have
	\begin{equation}\label{eq:21}
S \mathscr{R}^{n+1}=Q\left(\mathscr{R}^{n}-\sum_{k=1}^{n}w_{k}(\mathscr{R}^{n-k+1}-\mathscr{R}^{n-k})\right)+B,
\end{equation}
where
\begin{equation*}
S=\begin{pmatrix}
\mathscr{E}_{1}&\mathscr{E}_{2}&0&0&. .&0&0&0&0\\
\chi_{1}&\chi_{2}&\chi_{3}&0&. .&0 &0&0&0\\
0&\chi_{1}&\chi_{2}&\chi_{3}&. .&0&0&0&0\\
& & & & ..& & &\\
0&0 &0 &0 & ..&\chi_{1}&\chi_{2}&\chi_{3} &0\\
0&0&0&0&..&0&\chi_{1} &\chi_{2}&\chi_{3}\\
0&0&0&0&..&0&0&-\mathscr{E}_{2}&\mathscr{E}_{3}\\
\end{pmatrix},
\end{equation*}
\begin{eqnarray*}
	\mathscr{R}^{n}=\begin{pmatrix}
		&\mathscr{R}_{0}^{n}&\\
		&\mathscr{R}_{1}^{n}&\\
		&		.&\\
		&		.&\\
		&		.&\\		
		&	\mathscr{R}_{M-1}^{n}	&\\
		&		\mathscr{R}_{M}^{n}&
	\end{pmatrix},
	~~~~~~~~~~~	Q=\begin{pmatrix}
		0 &0&0&0&. . . . . &0&0&0&0\\
		\eta&1&\eta&0&. . . . . &0&0&0&0\\
		
		0&\eta & 1&\eta & .....&0& 0&0 &0\\
		
		& & & & .....& & &\\
		0&0 & 0&0 & .....&\eta&1& \eta& 0&\\
		0&0&0&0&.....&0&\eta &1&\eta\\
		0&0&0&0&.....&0&0&0&0\\
	\end{pmatrix},
\end{eqnarray*}
\begin{eqnarray*}
	\mathscr{E}_{1}&=&\alpha \widetilde{\gamma}\bar{\eta}\left(2+\frac{1}{\eta}\right)-\frac{\beta \widetilde{\gamma}\tilde{e}(\tilde{c}-1)}{\eta} , 
	\\ \mathscr{E}_{2}&=&-2\beta \widetilde{\gamma}\tilde{e}(\tilde{c}-1),\\
	\mathscr{E}_{3}&=&\alpha \widetilde{\gamma}\bar{\eta}\left(2+\frac{1}{\eta}\right)+\frac{\beta \widetilde{\gamma}\tilde{e}(\tilde{c}-1)}{\eta},
\end{eqnarray*}
and
\begin{equation*}
B=\begin{pmatrix}
&\mathscr{H}_{n}-\sum_{k=1}^{n}w_{k}(\mathscr{H}_{n-k+1}-\mathscr{H}_{n-k})-\frac{\chi_{1}}{\eta}\mathscr{H}_{n+1}+\widetilde{\gamma}\psi^{n+1}(\varkappa_0)&\\
&\widetilde{\gamma}\psi^{n+1}(\varkappa_1)&\\
&		\widetilde{\gamma}\psi^{n+1}(\varkappa_2)&\\
&		.&\\
&		.&\\
&		.&\\
&\widetilde{\gamma}\psi^{n+1}(\varkappa_{N_{\varkappa}-1})&\\
&\mathscr{G}_{n}-\sum_{k=1}^{n}w_{k}(\mathscr{G}_{n-k+1}-\mathscr{G}_{n-k})-\frac{\chi_{3}}{\eta}\mathscr{G}_{n+1}+\widetilde{\gamma}\psi^{n+1}(\varkappa_{N_{\varkappa}})&\\
\end{pmatrix}.
\end{equation*}
We observe that the set of systems in (\ref{eq:21}) for $n = 0,1,\cdots,N_{t},$ can be solved recursively if we know the initial vector $\mathscr{R}^{0}$. Note that $\mathscr{R}_{-1}^{0}$ and $\mathscr{R}_{N_{\varkappa}+1}^{0}$ can be then calculated using equations (\ref{eq:19}) and (\ref{eq:20}) respectively.
\subsection{\emph{\textbf{Initial state }}}\label{sub:4}
		\noindent To start the process, an appropriate initial vector $\mathscr{R}^{0}$ is needed for the system. For this, we consider the initial conditions in (\ref{eq4}),
	\begin{eqnarray*}
	\mathscr{U}^{0}_{\varkappa}(I_{p})=z'(I_{p}), ~~~~~~~~\mathscr{U}^{0}_{\varkappa}(F_{p})=z'(F_{p}).
	\end{eqnarray*}	
Now substituting the relation (\ref{eq:16}) in the above two equations we get	
	\begin{eqnarray}
	\mathscr{U}^{0}_{\varkappa}(\varkappa_{0})=\mathscr{U}^{0}_{\varkappa}(I_{p})=\tilde{e}(\tilde{c}-1) [\mathscr{R}_{1}^{0}- \mathscr{R}_{-1}^{0}] =z'(I_{p}),\label{eq:22}\\
	\mathscr{U}^{0}_{\varkappa}(\varkappa_{N_{\varkappa}})=\mathscr{U}^{0}_{\varkappa}(F_{p})=\tilde{e}(\tilde{c}-1) [\mathscr{R}_{N_{\varkappa}+1}^{0}- \mathscr{R}_{N_{\varkappa}-1}^{0}] =z'(F_{p}).\label{eq:23}
	\end{eqnarray}	
	Further, the relation (\ref{eq:15}) with the initial condition in (\ref{eq4}), yields an algebraic system of $(N_{\varkappa}+1)$ equations
	\begin{eqnarray}\label{eq:24}
\mathscr{U}^{0}(\varkappa_{m})=\eta \mathscr{R}_{m-1}^{0}+\mathscr{R}_{m}^{0}+\eta \mathscr{R}_{m+1}^{0}=z_{m},~~~~m=0,1,\cdots,N_{\varkappa},
	\end{eqnarray}	
	with unknowns $\mathscr{R}_{-1}^{0},\mathscr{R}_{0}^{0},\mathscr{R}_{1}^{0},\dots,\mathscr{R}_{N_{\varkappa}-1}^{0},\mathscr{R}_{N_{\varkappa}}^{0},\mathscr{R}_{N_{\varkappa}+1}^{0}$. Here,
	$\mathscr{R}_{-1}^{0}$ and $\mathscr{R}_{N_{\varkappa}+1}^{0}$ can be removed using equations (\ref{eq:22}) and (\ref{eq:23}) respectively. Thus, we get a tridiagonal system of size $(N_{\varkappa}+1) \times (N_{\varkappa}+1)$ which can also be solved using the Thomas algorithm. We have
	\begin{equation}\label{eq:25}
	T\mathscr{R}^{0}=\mathscr{C},
	\end{equation}
	where
	\begin{eqnarray*}
		T=\begin{pmatrix}
			1 &2\eta&0&0&. . . . . &0&0&0&0\\
		\eta&1&\eta&0&. . . . . &0&0&0&0\\
			
			0&\eta & 1&\eta & .....&0& 0&0 &0\\
			
			& & & & .....& & &\\
			0&0 & 0&0 & .....&\eta&1& \eta& 0&\\
			0&0&0&0&.....&0&\eta &1&\eta\\
			0&0&0&0&.....&0&0&2\eta&1\\
		\end{pmatrix},~~~~~~~~~~~~~~~~~~\\ 
		\\
		\\
		\mathscr{R}^{0}=\begin{pmatrix}
			&\mathscr{R}_{0}^{0}&\\
			&\mathscr{R}_{1}^{0}&\\
			&		.&\\
			&		.&\\
			&		.&\\		
			&	\mathscr{R}_{N_{\varkappa}-1}^{0}	&\\
			&		\mathscr{R}_{N_{\varkappa}}^{0}&
		\end{pmatrix},
		~~~~~~~~~~	\mathscr{C}=\begin{pmatrix}
			&z_{0}+\frac{\eta}{\tilde{e}(\tilde{c}-1)}z'(I_{p})&\\
			&	z_{1}&\\
			&		.&\\
			&		.&\\
			&		.&\\		
			&		z_{N_{\varkappa}-1}	&\\
			&			z_{N_{\varkappa}}-\frac{\eta}{\tilde{e}(\tilde{c}-1)}z'(F_{p})&
\end{pmatrix}.
\end{eqnarray*}

\section{Stability analysis}\label{sec:3}	
In this section we discuss the stability analysis of the proposed scheme.
\begin{theorem}\label{4:1thm:1}
	\noindent The numerical scheme (\ref{eq:18}) solving the TFBSM (\ref{eq:8})-(\ref{eq:9}) is unconditionally stable. 
\end{theorem}
\begin{proof}
	Let $\bar{\mathscr{R}}$ be a perturbed solution of the system (\ref{eq:18}). We will investigate how
	the perturbation $\delta_{m}^{n}=	\mathscr{R}_{m}^{n}-\bar{\mathscr{R}}_{m}^{n}$ develops over time. Note that  $\delta_{m}^{n}$ solves the following equation
\begin{eqnarray*}
	\chi_{1}\delta_{m-1}^{n+1}+\chi_{2}\delta_{m}^{n+1}+\chi_{3}\delta_{m+1}^{n+1}=\eta \delta_{m-1}^{n}+\delta_{m}^{n}+\eta \delta_{m+1}^{n}- \sum_{k=1}^{n}w_{k}\left[(\eta\delta_{m-1}^{n-k+1}+\delta_{m}^{n-k+1}+\eta\delta_{m+1}^{n-k+1})\right.\\-(\eta\delta_{m-1}^{n-k}+\delta_{m}^{n-k}+\eta\delta_{m+1}^{n-k})\big].~~~~~~~~~~~~~~~~~~~~~~~~~~~~~~~~~~~~~~~~~
\end{eqnarray*}
	\begin{equation}\label{eq:26}
	\end{equation}
	Now, to use the von Neumann stability analysis we assume that \begin{equation}\label{eq:27}
	\delta_{m}^{n}=\zeta^{n}e^{i\omega mh_{\varkappa}},
	\end{equation}
	where $i=\sqrt {-1}$ and $\omega$ is the wave number. Inserting the equation (\ref{eq:27}) in (\ref{eq:26}) yields
	\begin{equation}\label{eq:28}
	\zeta^{n+1}=\frac{\Upsilon_{1}}{(\Upsilon_{1}+\Upsilon_{2}+\Upsilon_{3}-i \Upsilon_{4})} \left[\zeta^{n}-\sum_{k=1}^{n} w_{k}(\zeta^{n-k+1}-\zeta^{n-k})\right],
	\end{equation}
	where
	\begin{equation*}
	\Upsilon_{1}=1+2\eta \cos\omega h_{\varkappa}, \hspace{0.9cm} 	\Upsilon_{2}=2\alpha \widetilde{\gamma} \bar{\eta}(1-\cos\omega h_{\varkappa}),	
	\end{equation*}
	\begin{equation*}
	\Upsilon_{3}=\gamma \widetilde{\gamma} (1+2\eta \cos\omega h_{\varkappa}), \hspace{0.5cm} \mbox{and } \hspace{0.5cm}	\Upsilon_{4}=2\beta \widetilde{\gamma} \tilde{e}(\tilde{c}-1)\sin\omega h_{\varkappa}.	
	\end{equation*}
	From equation (\ref{eq:28}), we have
		\begin{equation}\label{eq:29}
	|\zeta^{n+1}|^2 \leq \frac{|\Upsilon_{1}|^2}{((\Upsilon_{1}+\Upsilon_{2}+\Upsilon_{3})^{2}+ \Upsilon_{4}^2)} \left|\left[\zeta^{n}-\sum_{k=1}^{n} w_{k}(\zeta^{n-k+1}-\zeta^{n-k})\right]\right|^2.
\end{equation}
Since $|\Upsilon_{1}|^2 < (\Upsilon_{1}+\Upsilon_{2}+\Upsilon_{3})^{2}+ \Upsilon_{4}^2$, the equation (\ref{eq:29}) takes the form
\begin{equation}\label{ext}
|\zeta^{n+1}| \leq \left|\left[\zeta^{n}-\sum_{k=1}^{n} w_{k}(\zeta^{n-k+1}-\zeta^{n-k})\right]\right|.
\end{equation}
To show $|\zeta^{n+1}|\leq |\zeta^{0}|$, the mathematical induction is used. For $n=0$ in equation (\ref{ext}) we have $ |\zeta^{1}|\leq|\zeta^{0}|$.  Further, we assume that
\begin{eqnarray}\label{30}
|\zeta^{j}|\leq |\zeta^{0}|,\hspace{0.5cm} j=1,2,\dots,n.
\end{eqnarray} 
Now from the equation (\ref{ext}) and using the assumptions in (\ref{30}) and Lemma \ref{lemo} we have

\begin{equation*}
|\zeta^{n+1}| \leq 	\left|\left[w_{n}\zeta^{0}+\sum_{k=0}^{n-1} (w_{k}-w_{k+1})\zeta^{n-k}\right]\right| \leq \left[w_{n}+\sum_{k=0}^{n-1} (w_{k}-w_{k+1})\right]|\zeta^{0}|=|\zeta^{0}|.
\end{equation*}	
Hence, the inequality (\ref{30}) holds for all $n$. Thus, we observe that the perturbation is bounded
at each time level and therefore the numerical scheme (\ref{eq:18})  solving the TFBSM (\ref{eq:8})\--(\ref{eq:9}) is unconditionally stable.
This completes the proof.
\end{proof}

	\section{Convergence analysis}\label{sec:4}
We now discuss the convergence analysis of the numerical scheme (\ref{eq:18}). We shall use the following results.
\begin{Lemma}\label{4:2:lem:1}
	The basis elements $\{\mathscr{Q}_{m}(\varkappa)\}_{m=-1}^{N_{\varkappa}+1}$ of the exponential B-spline space $\mathscr{W}_{N_{\varkappa}+3}$ satisfy the following inequality
	\begin{equation*}
	\sum_{m=-1}^{N_{\varkappa}+1}|\mathscr{Q}_{m}(\varkappa)| \leq \frac{5}{2}, \hspace{1cm} I_{p} \leq \varkappa \leq F_{p}.
	\end{equation*}
\end{Lemma}
\begin{proof}
	See \cite[Lemma 4.1]{rao2008exponential} for the proof.
\end{proof}
\begin{theorem}\label{4:2:thm:1}
	Let the exact solution $u^{n+1}(\varkappa)$ of problem (\ref{eq:8})-(\ref{eq:9}) be interpolated by a unique exponential B-spline $\widetilde{V}^{n+1}(\varkappa)$ in $\mathscr{W}_{N_{\varkappa}+3}$. If $u\in C^{4,0}[0,1]$ and $\psi\in C^{2,0}[0,1],$ then there exists positive constants $c_{i}$ such that
	\begin{equation*}
	\left\vert\left\vert\frac{\partial^{i}}{\partial \varkappa^{i}}\left(u^{n+1}(\varkappa)-\widetilde{V}^{n+1}(\varkappa)\right)\right\vert\right\vert_{\infty} \leq c_{i} h_{\varkappa}^{4-i}, \hspace{0.6cm} i=0,1,2.
	\end{equation*}
\end{theorem}
\begin{proof}
	See \cite{pruess1976properties} for the proof.
\end{proof}
\begin{theorem}\label{4:2:thm:2}
	Let $\mathscr{U}^{n+1}(\varkappa)$ be the collocation approximation in (\ref{eq:14})  from the exponential B-spline space $\mathscr{W}_{N_{\varkappa}+3}$ to the exact solution $u^{n+1}(\varkappa)$ of the TFBSM (\ref{eq:8})-(\ref{eq:9}). If  $u\in C^{4,0}[0,1]$ and $\psi\in C^{2,0}[0,1]$, then there exists a positive constant $\lambda^*$ independent of $h_{\varkappa}$ such that
	\begin{equation*}
	||u^{n+1}(\varkappa)-\mathscr{U}^{n+1}(\varkappa)||_{\infty}\leq \lambda^* h_{\varkappa}^{2},\hspace{0.5cm} \forall n\geq 0.
	\end{equation*}
\end{theorem}
\begin{proof}
Let  $\widetilde{V}^{n+1}(\varkappa)$ be the unique exponential B\--spline interpolant to the exact solution $u^{n+1}(\varkappa)$ of the problem  (\ref{eq:8})-(\ref{eq:9}) given by
	\begin{equation}\label{eq:31}
	\widetilde{V}^{n+1}(\varkappa)=\sum_{m=-1}^{N_{\varkappa}+1}\mathscr{B}_{m}^{n+1}\mathscr{Q}_{m}(\varkappa).
	\end{equation}
	At time level $n+1$, we can write
		\begin{equation}\label{eq5}
	{}_{0}^{C}D^{\mu}_{t}\widetilde{V}^{n+1}(\varkappa)= \alpha\frac{\partial^{2}\widetilde{V}^{n+1}(\varkappa)}{\partial \varkappa^2}+\beta\frac{\partial \widetilde{V}^{n+1}(\varkappa)}{\partial x}-\gamma \widetilde{V}^{n+1}(\varkappa)+\widetilde \psi^{n+1}(\varkappa), \hspace{0.2cm} (-1\leq n \leq N_{t}-1)
	\end{equation}
	with
	\begin{equation}\label{eq6}
	\left\{%
	\begin{array}{ll}
	\widetilde{V}^0(\varkappa)=z(\varkappa),\\
	\widetilde{V}^{n+1}(\varkappa_{0})=\mathscr{H}(t_{n+1}), & -1\leq n \leq N_{t}-1,\\
	\widetilde{V}^{n+1}(\varkappa_{N_{\varkappa}})=\mathscr{G}(t_{n+1}), & -1\leq n \leq N_{t}-1.
	\end{array}
	\right.
	\end{equation}

\noindent Now using the equation (\ref{eq:31}) in (\ref{eq5}) we have 
\begin{eqnarray*}
	\chi_{1}\mathscr{B}_{m-1}^{n+1}+\chi_{2}\mathscr{R}_{B}^{n+1}+\chi_{3}\mathscr{B}_{m+1}^{n+1}=\eta \mathscr{B}_{m-1}^{n}+\mathscr{B}_{m}^{n}+\eta \mathscr{B}_{m+1}^{n}- \sum_{k=1}^{n}w_{k}[(\eta\mathscr{B}_{m-1}^{n-k+1}+\mathscr{B}_{m}^{n-k+1}+\eta\mathscr{B}_{m+1}^{n-k+1})\\-(\eta\mathscr{B}_{m-1}^{n-k}+\mathscr{B}_{m}^{n-k}+\eta\mathscr{B}_{m+1}^{n-k})]+\widetilde{\gamma}\widetilde{\psi}^{n+1}(\varkappa_{m}),~~~~~~~~~~~~~~~~~~~
\end{eqnarray*}
\begin{equation}\label{eq:35}
~~~~~~~~~~~~~~~~~~~~~~~~~~~~~~~~~~~~~~~~~~~~~~~~~~~~~~~~~~~~~~~~~~~~~~~~~~~~~~~~~~~~~~~~~~~~~~~~~~~~(0\leq m \leq N_{\varkappa}, -1\leq n \leq N_{t}-1)
\end{equation}
and also the boundary conditions in (\ref{eq6}) together with (\ref{eq:31}) yield
\begin{eqnarray}
\eta\mathscr{B}_{-1}^{n+1}=\mathscr{H}_{n+1}-\mathscr{B}_{0}^{n+1}-\eta\mathscr{B}_{1}^{n+1},\label{eq:36}\\
\eta\mathscr{B}_{N_{\varkappa}+1}^{n+1}=\mathscr{G}_{n+1}-\mathscr{B}_{N_{\varkappa}}^{n+1}-\eta\mathscr{B}_{N_{\varkappa}-1}^{n+1}.\label{eq:37}
\end{eqnarray}	
Next, the subtraction of equation (\ref{eq:18}) from (\ref{eq:35}) gives 
\begin{eqnarray*}
	\chi_{1}\lambda_{m-1}^{n+1}+\chi_{2}\lambda_{m}^{n+1}+\chi_{3}\lambda_{m+1}^{n+1}=\eta \lambda_{m-1}^{n}+\lambda_{m}^{n}+\eta \lambda_{m+1}^{n}- \sum_{k=1}^{n}w_{k}[(\eta\lambda_{m-1}^{n-k+1}+\lambda_{m}^{n-k+1}+\eta\lambda_{m+1}^{n-k+1})\\-(\eta\lambda_{m-1}^{n-k}+\lambda_{m}^{n-k}+\eta\lambda_{m+1}^{n-k})]+\widetilde{\gamma}(\widetilde{\psi}_{m}^{n+1}-\psi_{m}^{n+1}),~~~~~~~~~~~~~~~~~~~
\end{eqnarray*}
\begin{equation}\label{eq:38}
~~~~~~~~~~~~~~~~~~~~~~~~~~~~~~~~~~~~~~~~~~~~~~~~~~~~~~~~~~~~~~~~~~~~~~~~~~~~~~~~~~~~~~~~~~~~~~~~~~~~(0\leq m \leq N_{\varkappa}, -1\leq n \leq N_{t}-1).
\end{equation}
Also, subtracting equations (\ref{eq:19}) and (\ref{eq:20}) from (\ref{eq:36}) and (\ref{eq:37}) respectively yield
\begin{eqnarray}
\eta\lambda_{-1}^{n+1}=-\lambda_{0}^{n+1}-\eta\lambda_{1}^{n+1}, \hspace{0.5cm} -1\leq n \leq N_{t}-1,\label{eq:39}\\
\eta\lambda_{N_{\varkappa}+1}^{n+1}=-\lambda_{N_{\varkappa}}^{n+1}-\eta\lambda_{N_{\varkappa}-1}^{n+1},\hspace{0.5cm} -1\leq n \leq N_{t}-1,\label{eq:40}
\end{eqnarray}
where   $\lambda_{m}^{n+1}=\mathscr{B}_{m}^{n+1}-\mathscr{R}_{m}^{n+1}$ for  $-1\leq m\leq N_{\varkappa}+1$ and $-1\leq n\leq N_{t}-1$. 

Now subtracting equation (\ref{eq7}) from equation (\ref{eq5}) and using Theorem \ref{4:2:thm:1} we have
\begin{equation*}
\widetilde{\psi}_{m}^{n+1}-\psi_{m}^{n+1}=O(h_{\varkappa}^2).
\end{equation*}
Thus, it follows from the above equation that
\begin{equation}\label{eq8}
|\widetilde{\psi}_{m}^{n+1}-\psi_{m}^{n+1}|\leq Mh_{\varkappa}^2,
\end{equation}
where $M=\frac{c_{0}h_{\varkappa}^{2}\widetilde{T}}{\Gamma(2-\mu)}+\alpha c_{2}+\beta c_{1}h_{\varkappa}+\gamma c_{0}h_{\varkappa}^{2}$.

Now, let us take $\widetilde{\lambda}^{n+1}=\underset{-1 \leq m \leq N_{\varkappa}+1}{\max}|\lambda_{m}^{n+1}|$.
	Also, the initial condition of the problem implies that $\widetilde{\lambda}^{0}=0$. At the first time level, i.e. for $n=0,$ the equation (\ref{eq:38}) can be written as
	
	\begin{eqnarray}\label{eq:41}
	\chi_{2}\lambda_{m}^{1}=-	\chi_{1}\lambda_{m-1}^{1}-\chi_{3}\lambda_{m+1}^{1}+\eta \lambda_{m-1}^{0}+\lambda_{m}^{0}+\eta \lambda_{m+1}^{0}+\widetilde{\gamma}(\widetilde{\psi}_{m}^{1}-\psi_{m}^{1}),
	\end{eqnarray}
~~~~~~~~~~~~~~~~~~~~~~~~~~~~~~~~~~~~~~~~~~~~~~~~~~~~~~~~~~~~~~~~~~~~~~~~~~~~~~~~~~~~~~~~$0\leq m \leq N_{\varkappa}$.\\
By using Taylor's series expansion, for sufficiently small $h_{\varkappa},$ we have 
	\begin{eqnarray}\label{eq:42}
	\chi_{2}\lambda_{m}^{1}=	\lambda_{m}^{1}(-\chi_{1}-\chi_{3})+\widetilde{\gamma}(\widetilde{\psi}_{m}^{1}-\psi_{m}^{1}).
	\end{eqnarray}
After rearranging the terms and taking the absolute values on both sides of equation (\ref{eq:42}) we have
	\begin{eqnarray}\label{eq:43}
|\lambda_{m}^{1}|=\left| \frac{\widetilde{\gamma}(\widetilde{\psi}_{m}^{1}-\psi_{m}^{1})}{(\chi_{1}+\chi_{2}+\chi_{3})}\right|\leq \frac{\widetilde{\gamma}Mh_{\varkappa}^2}{(2\eta+1)(1+\gamma \widetilde{\gamma})}\leq \widetilde{M}h_{\varkappa}^2, \hspace{0.5cm} 0\leq m \leq N_{\varkappa},
\end{eqnarray}
where $\widetilde{M}=\frac{M\widetilde{\gamma}}{(1+2\eta)(1+\gamma \widetilde{\gamma})}$.	Also, from the boundary conditions we have the following estimates for $\lambda_{-1}^{1}$ and $\lambda_{N_{\varkappa}+1}^{1}:$
\begin{equation}\label{eq:44}
|\lambda_{-1}^{1}| \leq \widetilde{k}h_{\varkappa}^{2} \hspace{0.5cm} \mbox{and } \hspace{0.5cm} |\lambda_{N_{\varkappa}+1}^{1}| \leq \widetilde{k}h_{\varkappa}^{2},
\end{equation}
where $\widetilde{k}$ is a constant independent of $h_{\varkappa}$.\\	
Now combining the inequalities (\ref{eq:43}) and (\ref{eq:44}), we have
	\begin{equation}\label{eq:45}
	\widetilde{\lambda}^{1} \leq M_{1}h_{\varkappa}^{2},
	\end{equation}
	where  $M_1=\max \{\widetilde{M},\widetilde{k} \}$.
	
	Next, we use the mathematical induction to prove that $ \widetilde{\lambda}^{n+1}\le r^{*} h_x^2,$ where $r^{*}$ is a positive constant independent of $h_x$. Therefore, for this purpose we assume that
	\begin{equation}\label{eq:46} 
 \widetilde{\lambda}^{j}\leq K_j h_x^2
	\end{equation}
	is true for $1\leq j \leq n$. Since in equation  (\ref{eq:45}) we have shown that result (\ref{eq:46}) is true for $j=1$. Now we will prove the result for $j=n+1$. So, let $\widetilde{K}=\underset{0 \leq j \leq n}{\max} K_{j}$ and write the equation (\ref{eq:38}) in the form
	\begin{eqnarray*}
		\chi_{2}\lambda_{m}^{n+1}=-\chi_{1}\lambda_{m-1}^{n+1}-\chi_{3}\lambda_{m+1}^{n+1}+w_n(\eta \lambda_{m-1}^{0}+\lambda_{m}^{0}+\eta \lambda_{m+1}^{0})+ \sum_{k=0}^{n-1}(w_{k}-w_{k+1})(\eta\lambda_{m-1}^{n-k}+\lambda_{m}^{n-k}+\eta\lambda_{m+1}^{n-k})\\+\widetilde{\gamma}(\widetilde{\psi}_{m}^{n+1}-\psi_{m}^{n+1}),~~~~~~~~~~~~~~~~~~~
	\end{eqnarray*}
	\begin{equation}\label{eq:47}
	~~~~~~~~~~~~~~~~~~~~~~~~~~~~~~~~~~~~~~~~~~~~~~~~~~~~~~~~~~~~~~~~~~~~~~~~~~~~~~~~~~~~~~~~~~~~~~~~~~~~(0\leq m \leq N_{\varkappa}).
	\end{equation}
By applying Taylor's series expansion, for sufficiently small $h_{\varkappa}$ we can have	
		\begin{eqnarray*}
		\chi_{2}\lambda_{m}^{n+1}=(-\chi_{1}-\chi_{3})\lambda_{m}^{n+1}+w_n(2\eta+1) \lambda_{m}^{0}+(2\eta+1)  \sum_{k=0}^{n-1}(w_{k}-w_{k+1})\lambda_{m}^{n-k}+\widetilde{\gamma}(\widetilde{\psi}_{m}^{n+1}-\psi_{m}^{n+1}),~~~~~~~~~~~~~~~~~~~
	\end{eqnarray*}
After rearranging the terms and taking the absolute values on both sides of above equation and using the assumption  (\ref{eq:46}) we have	
	\begin{eqnarray*}
	(\chi_{1}+\chi_{2}+\chi_{3})|\lambda_{m}^{n+1}|&\leq& w_n(2\eta+1) |\lambda_{m}^{0}|+(2\eta+1)  \sum_{k=0}^{n-1}(w_{k}-w_{k+1})|\lambda_{m}^{n-k}|+\widetilde{\gamma}|(\widetilde{\psi}_{m}^{n+1}-\psi_{m}^{n+1})|\\
		&\leq&\widetilde{K}(2\eta+1)\left( \sum_{k=0}^{n-1}\{(w_{k}-w_{k+1})+w_n\} \right)h_{\varkappa}^2+\widetilde{\gamma}|(\widetilde{\psi}_{m}^{n+1}-\psi_{m}^{n+1})|.\\ \mbox{Thus},~~	|\lambda_{m}^{n+1}|	&\leq&\frac{w_0\widetilde{K}(2\eta+1)h_{\varkappa}^2}{(\chi_{1}+\chi_{2}+\chi_{3})}+\frac{\widetilde{\gamma}|(\widetilde{\psi}_{m}^{n+1}-\psi_{m}^{n+1})|}{(\chi_{1}+\chi_{2}+\chi_{3})}.
	\end{eqnarray*}
Now using the inequality (\ref{eq8}) we have	
		\begin{eqnarray}\label{eq:48}
		|\lambda_{m}^{n+1}|	&\leq&\frac{\widetilde{K}h_{\varkappa}^2}{(1+\gamma \widetilde{\gamma})}+\frac{\widetilde{\gamma}M h_{\varkappa}^{2}}{(2\eta+1)(1+\gamma \widetilde{\gamma})}\leq M^*h_{\varkappa}^2, \hspace{0.5cm} 0\leq m \leq N_{\varkappa},
	\end{eqnarray}
	where $M^*=\frac{\widetilde{K}}{(1+\gamma \widetilde{\gamma})}+\frac{M\widetilde{\gamma} }{(2\eta+1)(1+\gamma \widetilde{\gamma})}$.
	
	Like before as in (\ref{eq:44}), we can obtain bounds for $|\lambda_{-1}^{n+1}| $ and $|\lambda_{N_{\varkappa}+1}^{n+1}|$. With the help of these two bounds along with (\ref{eq:48}), it can be inferred that there exists a constant $r^{*}$ independent of $h_{\varkappa}$ such that
	\begin{eqnarray}\label{eq:49}
  \widetilde{\lambda}^{n+1}\le r^{*} h_x^2, \hspace{0.5cm} 
	\end{eqnarray}
where $r^{*}=\max \{M^*,\widetilde{k} \}$.

	Now 
	\begin{equation*}
	\widetilde{V}^{n+1}(\varkappa)-\mathscr{U}^{n+1}(\varkappa)= \sum_{m=-1}^{N_{\varkappa}+1}(\mathscr{B}_{m}^{n+1}-\mathscr{R}_{m}^{n+1})\mathscr{Q}_{m}(\varkappa).
	\end{equation*}
	Using Lemma \ref{4:2:lem:1} and inequality (\ref{eq:49}) in above relation we have
	\begin{equation}\label{eq:50}
	||\widetilde{V}^{n+1}(\varkappa)-\mathscr{U}^{n+1}(\varkappa)||_{\infty} \leq \frac{5}{2}r^{*}h_{\varkappa}^{2}.
	\end{equation}
	Using the triangle inequality we get
	\begin{equation}\label{eq:51}
	||u^{n+1}(\varkappa)-\mathscr{U}^{n+1}(\varkappa)||_{\infty}\leq ||u^{n+1}(\varkappa)-\widetilde{V}^{n+1}(\varkappa)||_{\infty}+||\widetilde{V}^{n+1}(\varkappa)-\mathscr{U}^{n+1}(\varkappa)||_{\infty}.
	\end{equation}
	Thus, using Theorem \ref{4:2:thm:1} and inequality (\ref{eq:50}), we get
	\begin{equation*}
	||u^{n+1}(\varkappa)-\mathscr{U}^{n+1}(\varkappa)||_{\infty}\leq \lambda^* h_{\varkappa}^{2},\hspace{0.5cm} \forall n\geq 0,
	\end{equation*}
	where $\lambda^*=\frac{5}{2}r^{*}+c_{0}h_{\varkappa}^{2}$.\\
	Hence, the proof is complete. 
\end{proof}
\begin{theorem}\label{4:2:thm:3}
	The present numerical scheme (\ref{eq:18}) for the TFBSM (\ref{eq:8})-(\ref{eq:9}) is convergent of order  $O(h_{t}^{2-\mu} + h_{\varkappa}^{2})$.
\end{theorem}
\begin{proof} 
	Theorem \ref{4:2:thm:2} together with relation (\ref{eq:tr}) provides us the following result 
	\begin{equation*}
	||u(\varkappa,t)-\mathscr{U}^{n+1}(\varkappa)||_{\infty}\leq k_{2} h_{t}^{2-\mu} +k_{3}h_{\varkappa}^{2}  ,
	\end{equation*}
	where $ k_{2}$ and $ k_{3}$ are positive constants, and thus the proposed method is convergent of order $O(h_{t}^{2-\mu} + h_{\varkappa}^{2})$ .
\end{proof}

	\section{Numerical illustrations and applications }\label{sec:5}
	\noindent This section includes the numerical results for three test problems to examine the performance of the proposed numerical scheme. Furthermore, Example (\ref{eg:4}) is considered as an application of the proposed scheme to price several different options like a European call option, European put option, and European double barrier knock-out call option. If $u(\varkappa_{m},t_{n})$ and $\mathscr{U}_{m}^{n}$ are the exact and approximate solutions of problem (\ref{eq:8})-(\ref{eq:9}) respectively at the point $(\varkappa_{m},t_{n})$, then the accuracy of the numerical solution will be measured
	as follows
	\begin{eqnarray}
L_{\infty}(h_{\varkappa},h_{t})=||u(\varkappa_{m},t_{n})-\mathscr{U}_{m}^{n}||_{\infty}=\mathop {\max}_{1\leq n \leq N_{t}}\mathop {\max}_{1\leq m \leq N_{\varkappa}-1}|u(\varkappa_{m},t_{n})-\mathscr{U}_{m}^{n}|\label{eq:000}\\
L_{2}(h_{\varkappa},h_{t})=||u(\varkappa_{m},t_{n})-\mathscr{U}_{m}^{n}||_{2}=\mathop {\max}_{1\leq n \leq N_{t}}\sqrt{h_{\varkappa}\sum_{m=1}^{N_{\varkappa}-1}(u(\varkappa_{m},t_{n})-\mathscr{U}_{m}^{n})^{2}}\label{eq:301}
	\end{eqnarray}
The following formula is used to evaluate the order of convergence (EOC)
	\begin{equation}\label{eq:302}
	\displaystyle \emph{EOC} = \begin{cases} 
\log_{2}\left(\frac{L_{j}(h_{\varkappa},(h_{t})_{1})}{L_{j}(h_{\varkappa},(h_{t})_{2})}\right), & \text{in time,} \\  
\log_{2}\left(\frac{L_{j}((h_{\varkappa})_{1},h_{t})}{L_{j}((h_{\varkappa})_{2},h_{t})}\right), & \text{in space},  
	\end{cases}
	\end{equation}
	where $j=2$ or $\infty$.
		\begin{example}\label{eg:1}  \emph{\cite{zhang2016numerical} Let us consider the problem (\ref{eq:8})\--(\ref{eq:9}) on the domain  $( 0, 1 ) \times(0,1)$,} 
			\begin{equation*}
			{}_{0}^{C}D^{\mu}_{t}u(\varkappa,t)= \alpha\frac{\partial^{2}u(\varkappa,t)}{\partial \varkappa^2}+\beta\frac{\partial u(\varkappa,t)}{\partial \varkappa}-\gamma u(\varkappa,t)+\psi(\varkappa,t), 
			\end{equation*}
			\emph{with}
			\begin{equation*}
			\left\{%
			\begin{array}{ll}
			u(\varkappa,0)=\varkappa^2(1-\varkappa),\\	
			u(0,t)=0,\\
			u(1,t)=0,
			\end{array}%
			\right.~~~~~~~~~~~~~~~~~~~~~~~~~~~~~~~~~~~~~~~~~~~~~~~~~~~~~~~~~~~~~~~~~~~~~~~~~~~~~~~~~~
			\end{equation*}
		\emph{and the source term}
		\begin{equation*}\psi(\varkappa,t)=\frac{2}{\Gamma(3-\mu)}t^{2-\mu}\varkappa^{2}(1-\varkappa)+\frac{2}{\Gamma(2-\mu)}t^{1-\mu}\varkappa^{2}(1-\varkappa)-(t+1)^{2}(\alpha(2-6\varkappa)+\beta \varkappa(2-3\varkappa)-\gamma \varkappa^{2}(1-\varkappa)).
		\end{equation*}	
		\emph{The exact solution of this test problem is }  $u(\varkappa,t)=(t+1)^{2}\varkappa^{2}(1-\varkappa)$. \emph{We will solve this problem with pre\--mentioned values of parameters} $r_{f}=0.05$, $\sigma=0.25$. 
	\end{example}

		
		
		
	\begin{table}
	\caption{ Errors $L_{2}$ and $L_{\infty}$ and corresponding orders of convergence with $\mu=0.5$, $\rho=1.5$ and $h_{\varkappa}=0.002$ for Example \ref{eg:1}}
	\begin{tabular}{ |p{2cm} p{2cm} p{2cm} p{2cm} p{2cm} p{2cm}  p{2cm}|}		\multicolumn{7}{c}{} \\
		\hline
		$h_{t}$& $\frac{1}{10}$& $\frac{1}{20}$&$\frac{1}{40}$&$\frac{1}{80}$&$\frac{1}{160}$&$\frac{1}{320}$\\
		\hline
		$L_{2}$&   1.0584e-03 & 3.8720e-04&  1.3997e-04&5.0197e-05&  1.7890e-05&  6.3288e-06	\\
		EOC&$\--$&1.4507&1.4680&1.4794&1.4885&1.4991\\
		$L_{\infty}$	&  1.5570e-03& 5.6937e-04&2.0577e-04& 7.3779e-05& 2.6286e-05& 9.2927e-06
		\\
		EOC	&$\--$	 &1.4513&1.4684&1.4797& 1.4889&1.5001\\	
		\hline
	\end{tabular}\label{5tab:1}
\end{table}
\begin{table}
	\caption{Comparison of the maximum error $L_{\infty}$ and the corresponding order of convergence with $\mu=0.7$, $N_{t}=1000$ and $\rho=1.5$ for Example \ref{eg:1}.}
	\begin{tabular}{ |p{1cm} p{2cm} p{1.2cm} p{2cm} p{1.2cm} p{2cm} p{1.2cm} p{2cm} p{1.2cm}|}		\multicolumn{8}{c}{} \\
		\hline
		$	N_{\varkappa}$	&  Present method& & 	Method in \cite{zhang2016numerical}&&Present method&&	Method in \cite{zhang2016numerical}& \\
		\hline
		& $L_{\infty}$-error &EOC &$L_{\infty}$-error &EOC& $L_{2}$-error&EOC&$L_{2}$-error&EOC\\
		\hline
		&  & & && &&&\\		
		4	&  0.002&\--&0.0030&\--&0.001& \--&0.0024&\--\\
		8	&   4.7739e-04 & 2.12&7.6750e-04&1.98& 2.8144e-04& 2.05&6.1678e-04&1.96\\
		16	&   1.1112e-04&2.10&1.8629e-04&2.04& 6.5497e-05& 2.10&1.5079e-04&2.03\\
		
		32	& 2.2337e-05 &2.31&4.0698e-05&2.19&1.2680e-05&2.37&3.2995e-05&2.19
		\\
		\hline
	\end{tabular}\label{5tab:2}
\end{table}
%
%
%
	\begin{table}
	\caption{ {$L_{2}$ error and corresponding order of convergence for various of $\mu$ with $\rho=1.5$ and $N_{t}=500$ for Example \ref{eg:1}.}}
	\begin{tabular}{ |p{1.5cm} p{2cm} p{2cm} p{2cm} p{2cm} p{2cm} p{2cm}|}		\multicolumn{6}{c}{} \\
		\hline
		$\mu$	&	$N_{\varkappa}$&$2^{2}$&  $2^{3}$&$2^{4}$ & $2^{5}$& $2^{6}$ \\
		\hline
		0.2	&	$L_{2}$&1.3129e-03& 3.2284e-04& 7.9810e-05&1.9753e-05& 4.7919e-06\\
		&	EOC&$\--$&2.0239& 2.0162&2.0145&2.0434\\
		
		0.4	&	$L_{2}$&1.2590e-03& 3.0844e-04&7.5439e-05&1.7916e-05&3.6469e-06\\
		&	EOC&$\--$& 2.0292&2.0316& 2.0741& 2.2965\\
		
	0.6	&	$L_{2}$& 1.1961e-03 & 2.8875e-04&6.6856e-05&  1.2634e-05&  2.6656e-06
		\\	
			&	EOC&$\--$&2.0504& 2.1107&2.4038& 2.2448\\
		\hline
	\end{tabular}\label{5tab:3}
\end{table}
	\begin{table}
		\caption{ Maximum error $L_{\infty}$ and corresponding order of convergence for various $\mu$ with $\rho=1.5$ and $N_{t}=500$ for Example \ref{eg:1}.}
		\begin{tabular}{ |p{1.5cm} p{2cm} p{2cm} p{2cm} p{2cm} p{2cm} p{2cm}|}		\multicolumn{7}{c}{} \\
			\hline
			$\mu$	&	$N_{\varkappa}$&$2^{2}$&  $2^{3}$&$2^{4}$& $2^{5}$& $2^{6}$ \\
			\hline
			
			0.2	&	$L_{\infty}$	&  2.3031e-03 & 5.4273e-04& 1.3376e-04& 3.3122e-05& 8.0812e-06   \\
			&EOC	&$\--$	& 2.0852 &2.0206&2.0137& 2.0352\\
		0.4	&$L_{\infty}$& 2.2197e-03&5.2005e-04& 1.2692e-04&3.0408e-05& 6.3578e-06
		 \\
			&EOC	&$\--$&2.0936& 2.0348&2.0614& 2.2579\\
			0.6	&	$L_{\infty}$	& 2.1229e-03 & 4.8913e-04& 1.1338e-04& 2.2242e-05& 5.2007e-06\\
&EOC&$\--$	&2.1177&2.1090&2.3499&2.0965 \\	
			\hline
		\end{tabular}\label{5tab:4}
	\end{table}
Table \ref{5tab:1} contains the errors $L_{2}$ and $L_{\infty}$ and respective orders of convergence, for Example \ref{eg:1} with $\rho=1.5$, $\mu=0.5$, $h_{\varkappa}=0.002$, and different time spacings $h_{t}$. The results of this table validate the theoretical order of convergence proved in Theorem \ref{4:2:thm:3}. Similarly,  $L_{2}$ and $L_{\infty}$ errors with corresponding orders of convergence are calculated for $\rho=1.5$, $N_{t}=500$ and multiple values of $\mu,$ and are shown in Tables \ref{5tab:3} and \ref{5tab:4}, respectively. From these tables, it can be viewed that as we increase the discretization points $N_{\varkappa}$, the errors decrease and the numerically calculated spatial order of convergence is $2$, which is in support of the result proved in Theorem \ref{4:2:thm:3}. 

Figures \ref{fig:2}, \ref{fig:6}, and \ref{fig:9} represent the exact and numerical solutions taking different fractional orders $\mu$ for Examples \ref{eg:1}, \ref{eg:2}, and \ref{eg:3} respectively. From these graphs, we can say that the proposed method approximates the TFBSM very well. The graphs shown in Figures \ref{fig:3}, \ref{fig:7}, and \ref{fig:10} compare the numerical and exact solutions of  Examples \ref{eg:1}, \ref{eg:2}, and \ref{eg:3} respectively, at different time levels. From these graphs, we observe that the numerical and exact solutions follow almost the same path.
Further, Figure \ref{fig:4} shows the three-dimensional plots of
the maximum errors for Example \ref{eg:1} with $\rho=1.5$, $\mu=0.9$, and different $N_{\varkappa},$ $N_{t}$. Moreover, for $\rho=8.6$, $\mu=0.99$, and different $N_{\varkappa},$ $N_{t}$ the three-dimensional graphics of the maximum errors for Example \ref{eg:3} are given in Figure \ref{fig:11} and
show that the errors decrease as the discretization points $N_{\varkappa}$ and $N_{t}$ increase.
\begin{figure}
	\centering
	\includegraphics[width=1\linewidth]{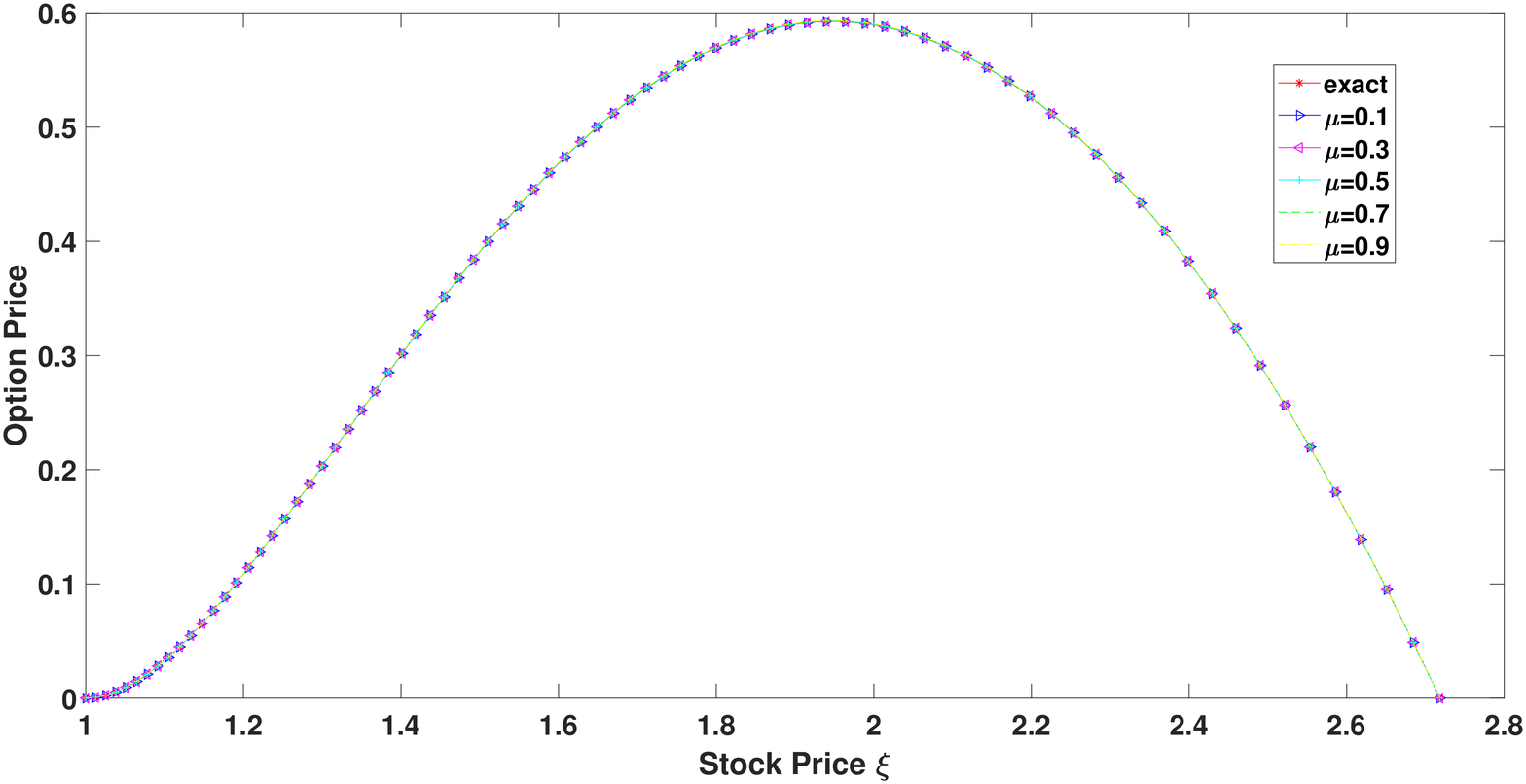}
	\caption{Exact and numerical solutions of Example \ref{eg:1} with $\rho=1.5$ and $N_{\varkappa}=N_{t}=80$.}\label{fig:2}
\end{figure}
	\begin{figure}
	\centering
	\includegraphics[width=1\linewidth]{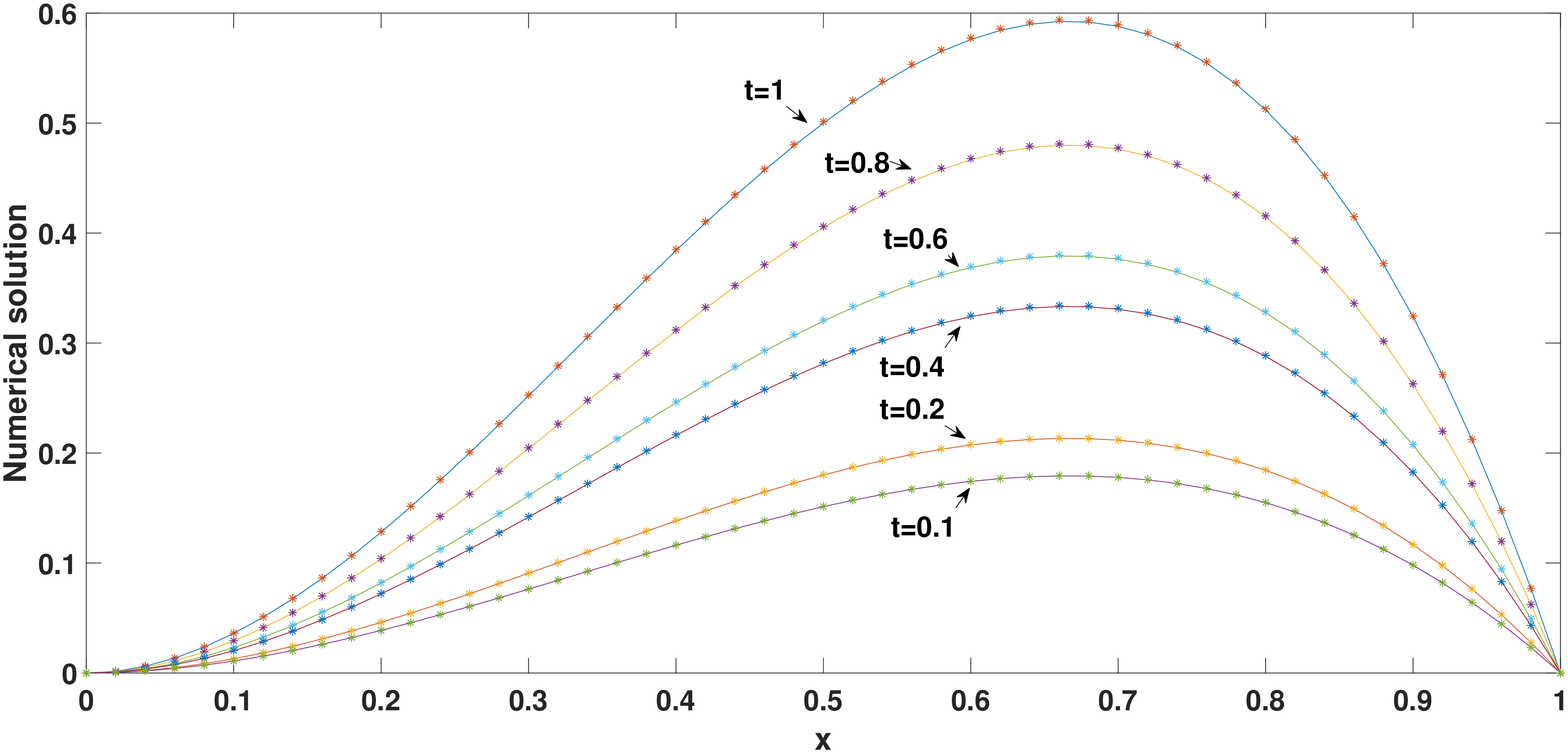}
	\caption{Numerical solution (starred line) and exact solution (solid line) of Example \ref{eg:1} with $\mu=0.9$, $\rho=1.5$ and $N_{\varkappa}=N_{t}=50,$ at different time levels.}\label{fig:3}
\end{figure}
	\begin{figure}
	\centering
	\begin{minipage}{.5\textwidth}
		\centering
		\includegraphics[width=1\linewidth]{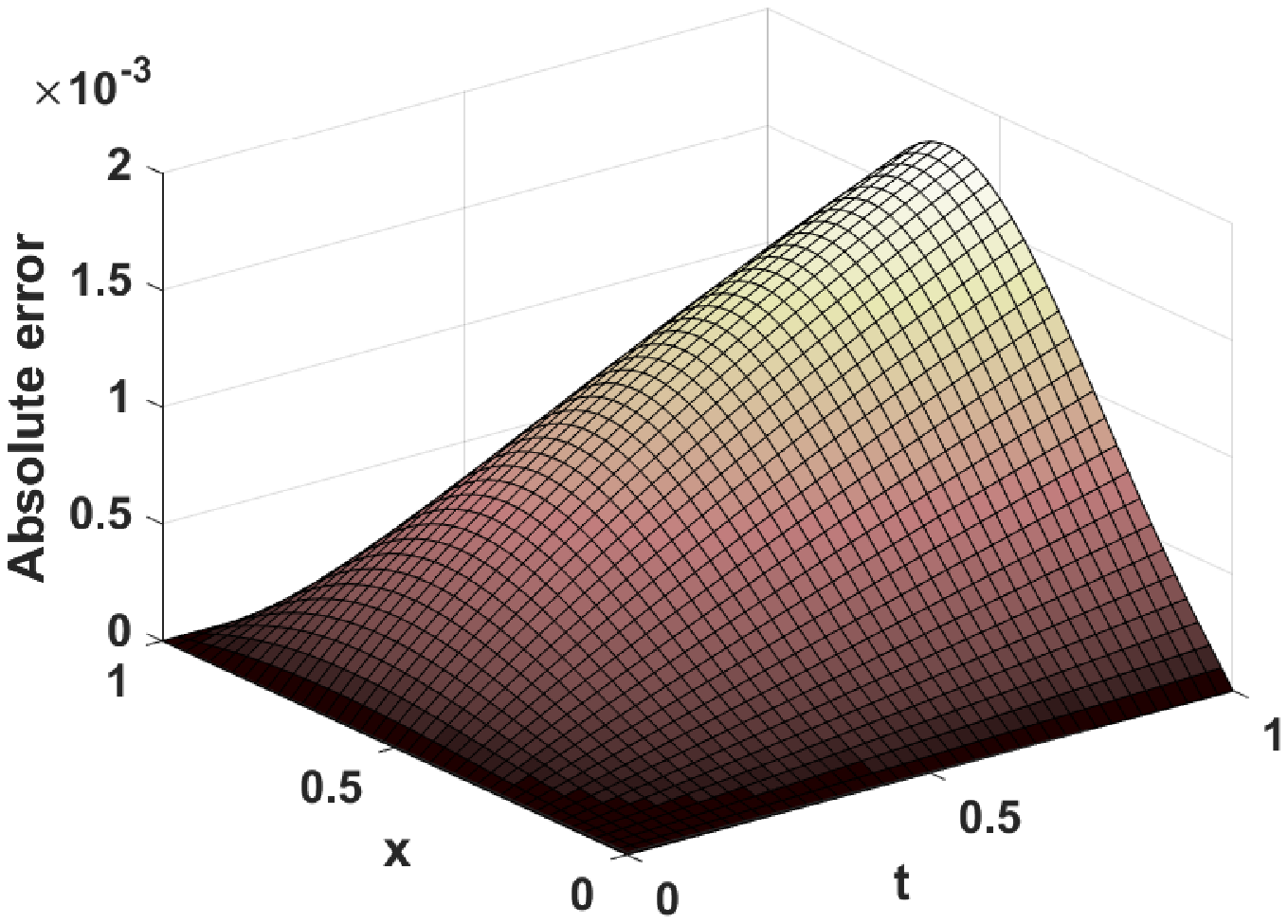}
		{$(a)\hspace{0.1cm} N_{\varkappa}=N_{t}=40$}
	\end{minipage}%
	\begin{minipage}{.5\textwidth}
		\centering
		\includegraphics[width=1\linewidth]{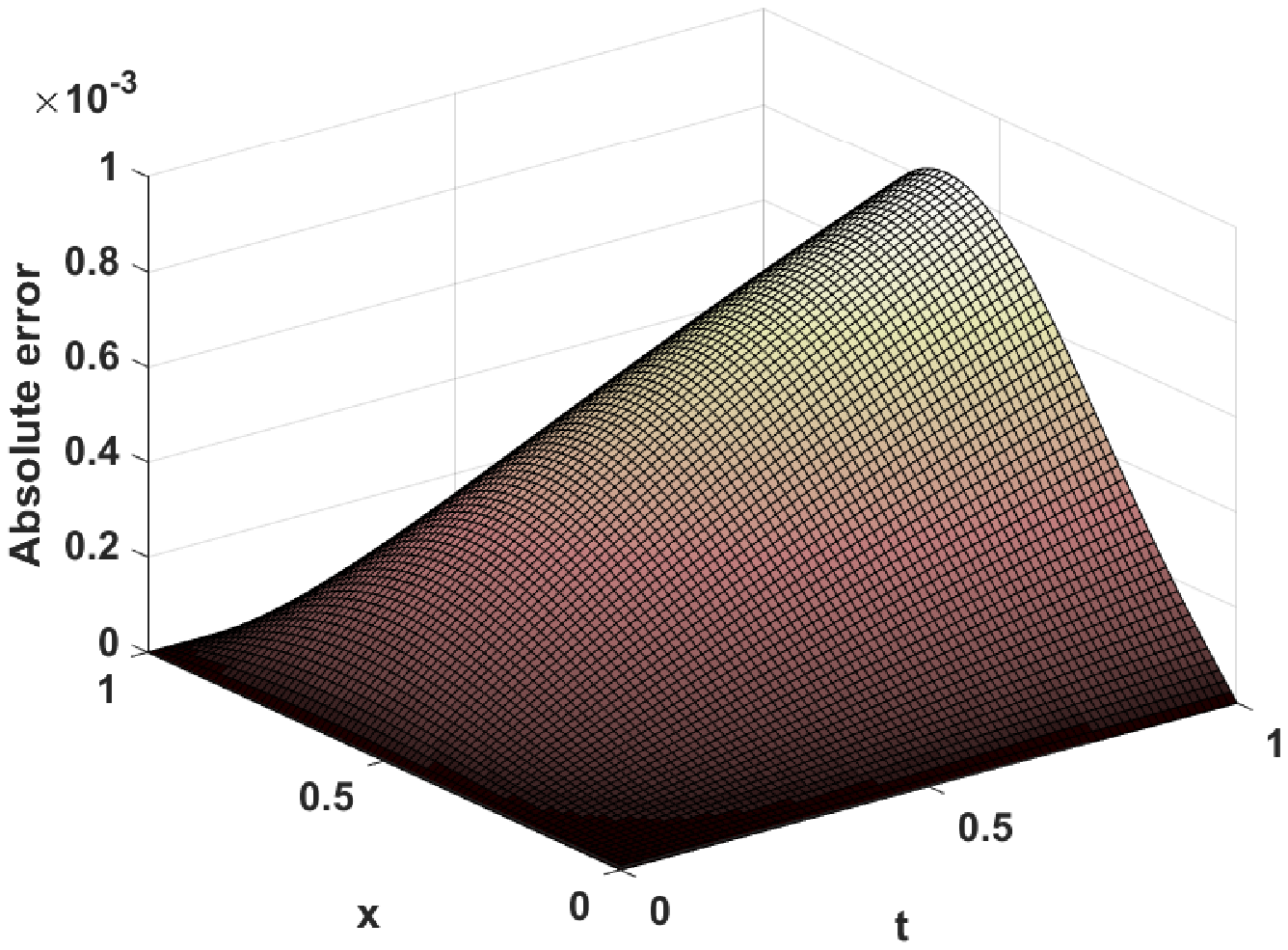}
		{$(b)\hspace{0.1cm} N_{\varkappa}=N_{t}=80$}
	\end{minipage}
\end{figure}
\begin{figure}
	\centering
	\begin{minipage}{.5\textwidth}
		\centering
		\includegraphics[width=1\linewidth]{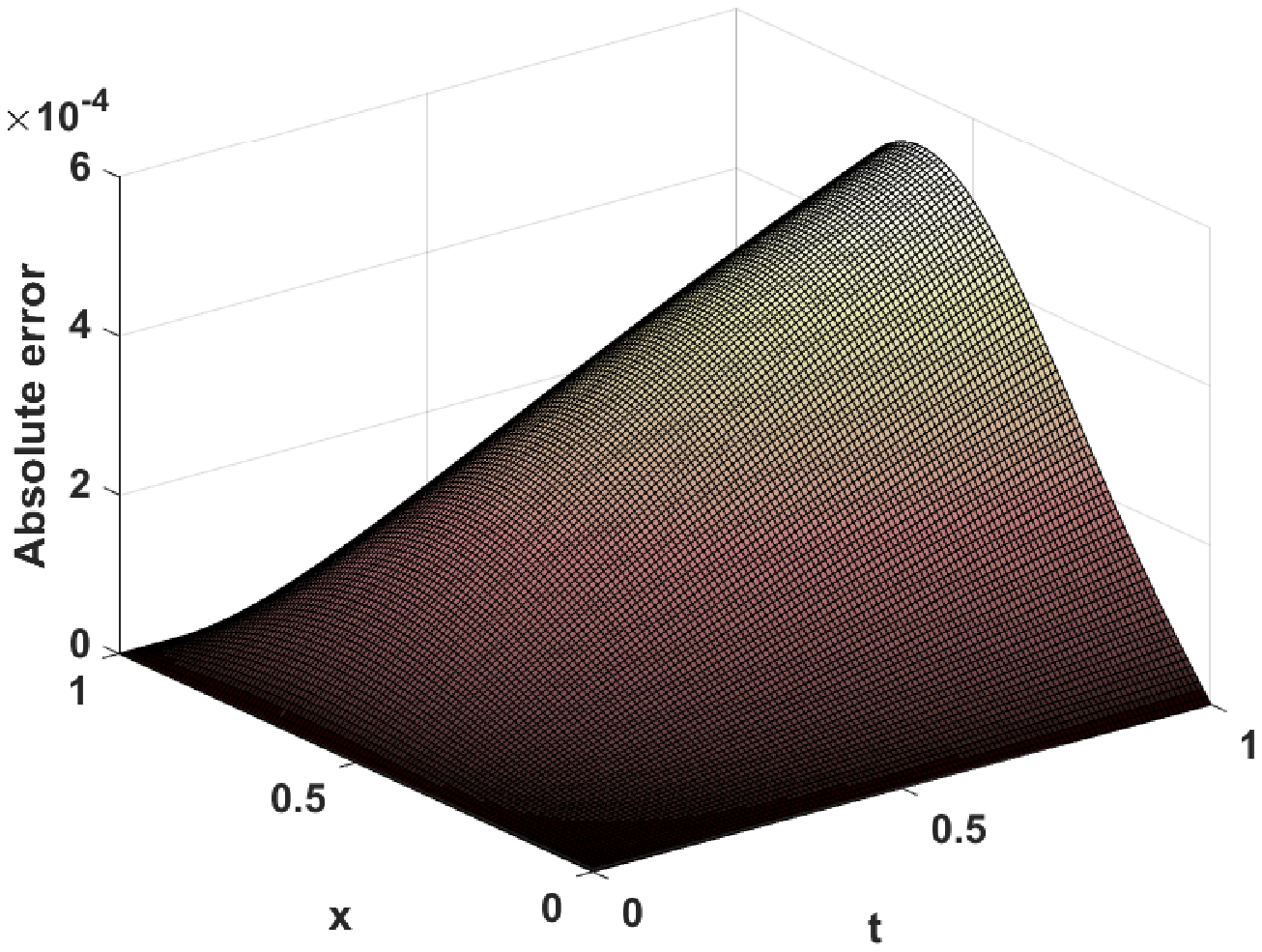}
		{$(c)\hspace{0.1cm} N_{\varkappa}=N_{t}=120$}
	\end{minipage}%
	\begin{minipage}{.5\textwidth}
		\centering
		\includegraphics[width=1\linewidth]{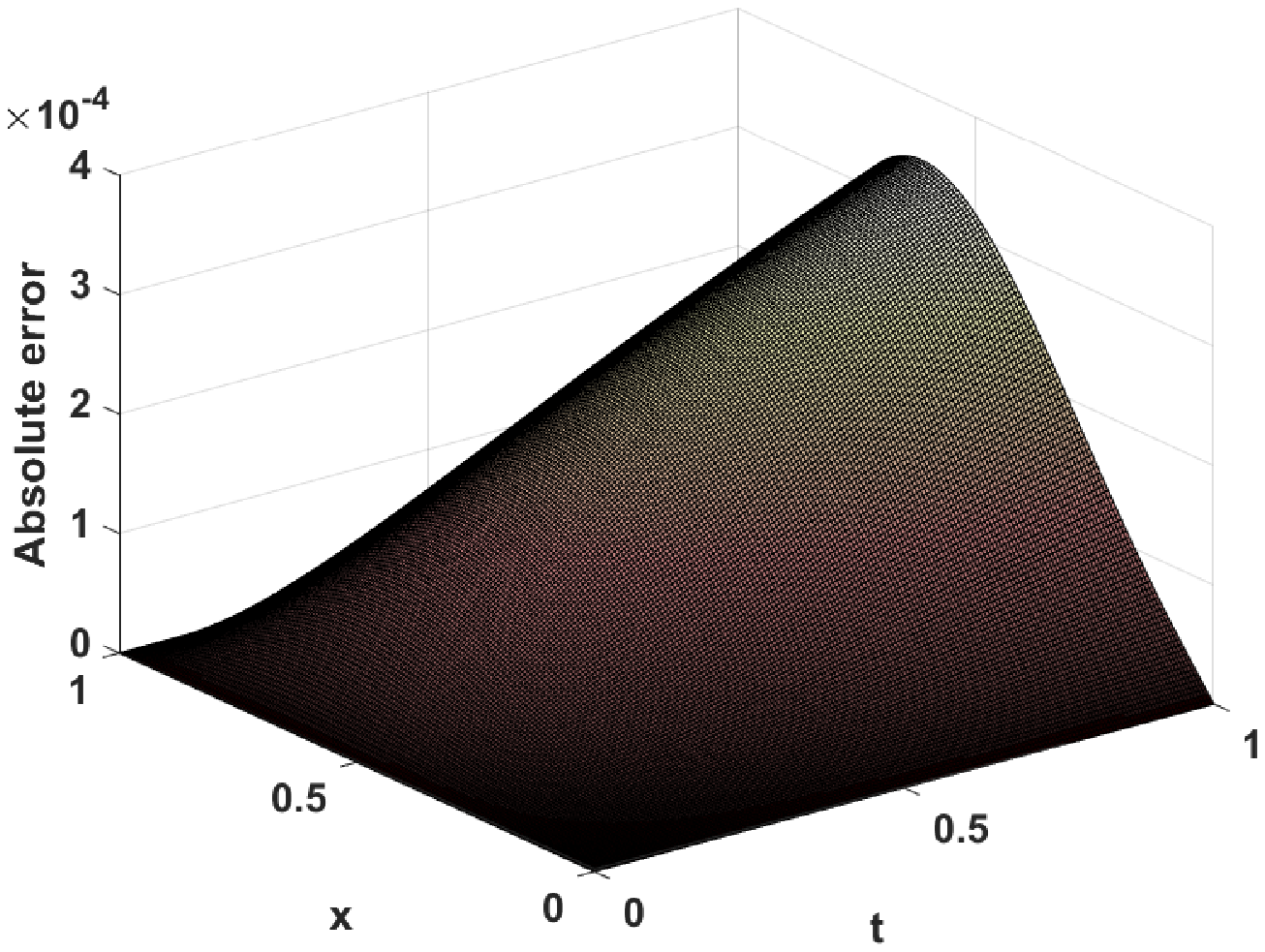}
		{$(d)\hspace{0.1cm} N_{\varkappa}=N_{t}=180$}
	\end{minipage}
	\caption{Maximum absolute error plots for different mesh sizes for Example \ref{eg:1} with $\mu=0.9$ and $\rho=1.5$.}\label{fig:4}
\end{figure}
	\begin{example}\label{eg:2}  \emph{Let us consider the problem (\ref{eq:8})\--(\ref{eq:9}) on the domain  $( 0, 1 ) \times(0,1),$} 
			\begin{equation*}
		{}_{0}^{C}D^{\mu}_{t}u(\varkappa,t)= \alpha\frac{\partial^{2}u(\varkappa,t)}{\partial \varkappa^2}+\beta\frac{\partial u(\varkappa,t)}{\partial \varkappa}-\gamma u(\varkappa,t)+\psi(\varkappa,t), 
		\end{equation*}
		\emph{with}
		\begin{equation*}
		\left\{%
		\begin{array}{ll}
		u(\varkappa,0)=1+\varkappa^2+\varkappa^3,\\	
		u(0,t)=(1+t)^2,\\
		u(1,t)=3(1+t)^2,
		\end{array}%
		\right.~~~~~~~~~~~~~~~~~~~~~~~~~~~~~~~~~~~~~~~~~~~~~~~~~~~~~~~~~~~~~~~~~~~~~~~~~~~~~~~~~~
		\end{equation*}
	\emph{ and the source term}
	\begin{equation*}\psi(\varkappa,t)=\left( \frac{2t^{2-\mu}}{\Gamma(3-\mu)}+\frac{2t^{1-\mu}}{\Gamma(2-\mu)}\right)(1+\varkappa^2+\varkappa^3)-(t+1)^{2}\left(\alpha(6\varkappa+2)+\beta \varkappa(2+3\varkappa)-\gamma(1+\varkappa^2+\varkappa^3)\right).
	\end{equation*}	
	\emph{The exact solution of this test problem is }  $u(\varkappa,t)=(t+1)^{2}(1+\varkappa^2+\varkappa^3)$. \emph{We will solve this problem with pre\--mentioned values of parameters} $r_{f}=0.5$, $\alpha=1$. 
\end{example}
The numerical results obtained by applying the proposed method for Example \ref{eg:2} are given in Tables \ref{5tab:5}, \ref{5tab:6}, \ref{5tab:7}, and \ref{5tab:8}. Tables \ref{5tab:5} and \ref{5tab:6}, give the errors $L_{2}$ and $L_{\infty}$ with the corresponding orders of convergence calculated for $\rho=0.5$, $N_{t}=2500$, and different fractional orders $\mu$. From these tables, we can see that the numerically evaluated spatial order of convergence comes out to be two. Similarly, from Tables \ref{5tab:8} and \ref{5tab:9}, we can see that the errors $L_{2}$ and $L_{\infty}$ decrease as we increase the discretization points $N_{\varkappa}$. Also, these tables show that the temporal order of convergence is $O(h_{t}^{2-\mu})$, which is consistent with Theorem \ref{4:2:thm:3}.
	\begin{table}
	\caption{ {$L_{2}$ error and corresponding order of convergence for various $\mu$ with $\rho=0.5$ and $N_{t}=2500$ for Example \ref{eg:2}.}}
	\begin{tabular}{ |p{2cm} p{2cm} p{2cm} p{2cm} p{2cm} p{2cm} p{2cm}|}		\multicolumn{6}{c}{} \\
		\hline
		$\mu$	&	$N_{\varkappa}$&$2^{2}$&  $2^{3}$&$2^{4}$ & $2^{5}$& $2^{6}$ \\
		\hline
		0.2	&	$L_{2}$&  2.0482e-03 & 5.1071e-04& 1.2755e-04&3.1892e-05&7.9881e-06
		\\
		&	EOC&$\--$&2.0037& 2.0015&1.9998&1.9973\\
		
		0.4	&	$L_{2}$&  2.0425e-03& 5.0936e-04&1.2732e-04& 3.1946e-05& 8.1137e-06\\
		&EOC&$\--$& 2.0036&2.0003&1.9947&1.9772\\
		
		0.6	&	$L_{2}$&2.0413e-03& 5.0980e-04&  1.2819e-04& 3.2936e-05& 9.1312e-06	
		\\
		&	EOC&$\--$& 2.0015&1.9916&1.9606&1.8508
		 \\	
		\hline
	\end{tabular}\label{5tab:5}
\end{table}
\begin{table}
	\caption{ Maximum absolute error $L_{\infty}$ and corresponding order of convergence for various $\mu$ with $\rho=0.5$ and $N_{t}=2500$ for Example \ref{eg:2}.}
	\begin{tabular}{ |p{1.5cm} p{2cm} p{2cm} p{2cm} p{2cm} p{2cm} p{2cm}|}		\multicolumn{7}{c}{} \\
		\hline
		$\mu$	&	$N_{\varkappa}$&$2^{2}$&  $2^{3}$&$2^{4}$& $2^{5}$& $2^{6}$ \\
		\hline
		
		0.2	&	$L_{\infty}$	& 2.7740e-03&  6.9807e-04& 1.7615e-04&4.4043e-05& 1.1038e-05 \\
		&EOC	&$\--$&   1.9905&1.9866& 1.9998&1.9965\\
		0.4	&$L_{\infty}$& 2.7661e-03& 6.9623e-04 &1.7582e-04&4.4114e-05& 1.1209e-05\\
		&EOC&$\--$&1.9902 &1.9855&1.9948&1.9765\\
	0.6	&	$L_{\infty}$&  2.7643e-03&6.9683e-04& 1.7701e-04& 4.5467e-05&1.2603e-05 \\
		&EOC	&$\--$&1.9880&1.9770&1.9610&1.8510 \\	
		\hline
	\end{tabular}\label{5tab:6}
\end{table}
	\begin{table}
	\caption{ $L_{2}$ and $L_{\infty}$ errors and corresponding orders of convergence with $\mu=0.3$, $\rho=0.5$ and $N_{\varkappa}=1000$ for Example \ref{eg:2}.}
	\begin{tabular}{ |p{2cm} p{2cm} p{2cm} p{2cm} p{2cm} p{2cm}  p{2cm}|}		\multicolumn{7}{c}{} \\
		\hline
		$N_{t}$& $10$& $20$&$40$&$80$&$160$&$320$\\
		\hline
		$L_{2}$&    6.1294e-04&   1.9775e-04&   6.3151e-05&   2.0026e-05&   6.3280e-06&   2.0061e-06	\\
	EOC&$\--$&   1.6321&1.6468&1.6570&1.6620&1.6574\\
		$L_{\infty}$	&    8.3939e-04 &   2.7082e-04&   8.6487e-05&   2.7426e-05  & 8.6669e-06&   2.7478e-06
		\\
	EOC	&$\--$	&   1.6320&1.6468&1.6569&1.6620&1.6572\\	
		\hline
	\end{tabular}\label{5tab:7}
\end{table}
	\begin{table}
	\caption{ $L_{2}$ and $L_{\infty}$ errors and corresponding orders of convergence with $\mu=0.7$, $\rho=0.5$ and $N_{\varkappa}=1000$ for Example \ref{eg:2}.}
	\begin{tabular}{ |p{2cm} p{2cm} p{2cm} p{2cm} p{2cm} p{2cm}  p{2cm}|}		\multicolumn{7}{c}{} \\
		\hline
		$N_{t}$& $10$& $20$&$40$&$80$&$160$&$320$\\
		\hline
		$L_{2}$&3.9975e-03&1.6365e-03& 6.6783e-04& 2.7203e-04& 1.1069e-04& 4.5023e-05
			\\
	EOC&$\--$& 1.2885&1.2931&1.2957&1.2972&1.2978\\
		$L_{\infty}$	&5.4792e-03& 2.2432e-03&9.1541e-04& 3.7288e-04& 1.5173e-04& 6.1715e-05
		\\
		EOC	&$\--$	& 1.2884& 1.2931& 1.2957&1.2972&1.2978\\	
		\hline
	\end{tabular}\label{5tab:8}
\end{table}
	\begin{figure}
	\centering
	\includegraphics[width=1\linewidth]{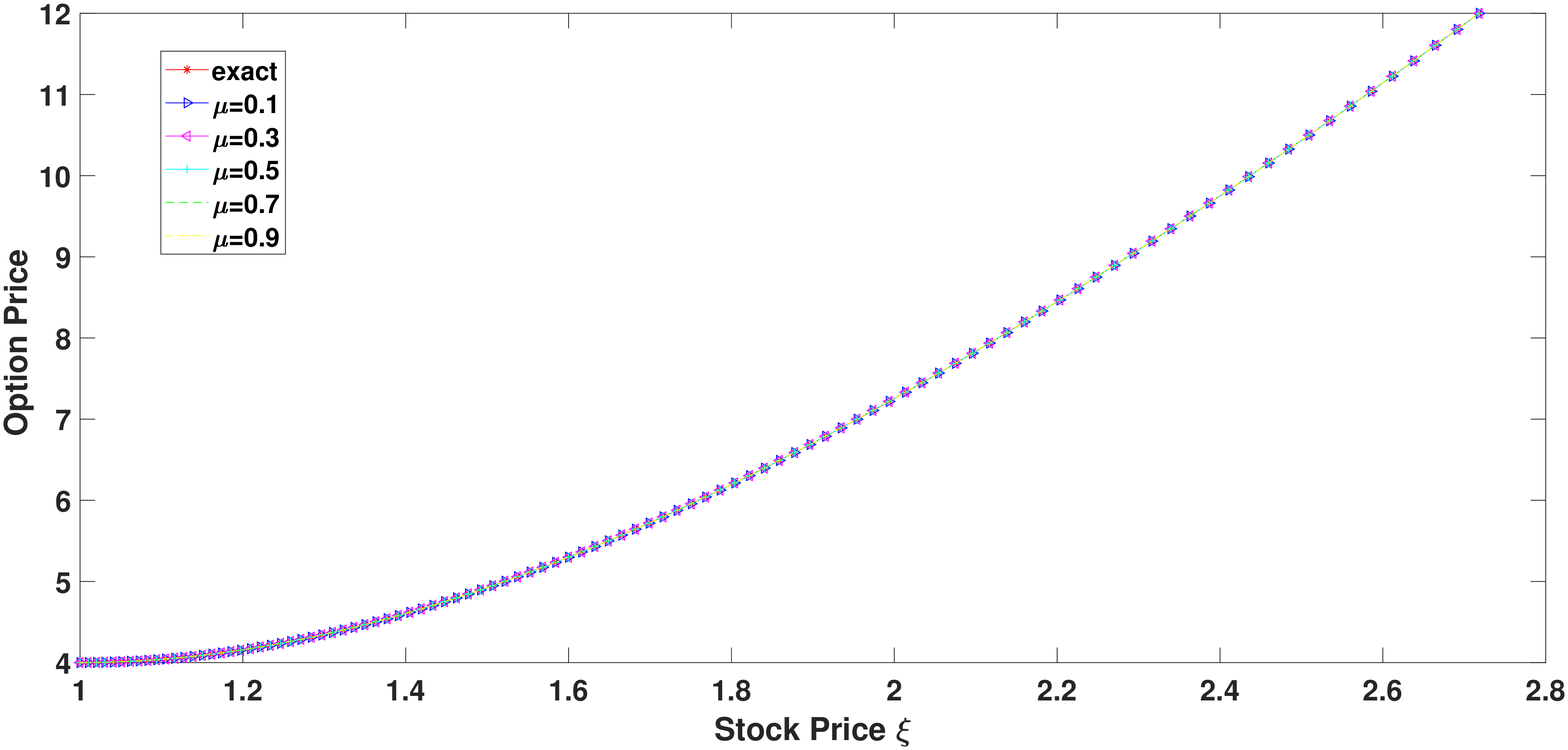}
	\caption{The exact and numerical solutions of Example \ref{eg:2} with $\rho=0.5$ and $N_{\varkappa}=N_{t}=100$.}\label{fig:6}
\end{figure}
	\begin{figure}
	\centering
	\includegraphics[width=1\linewidth]{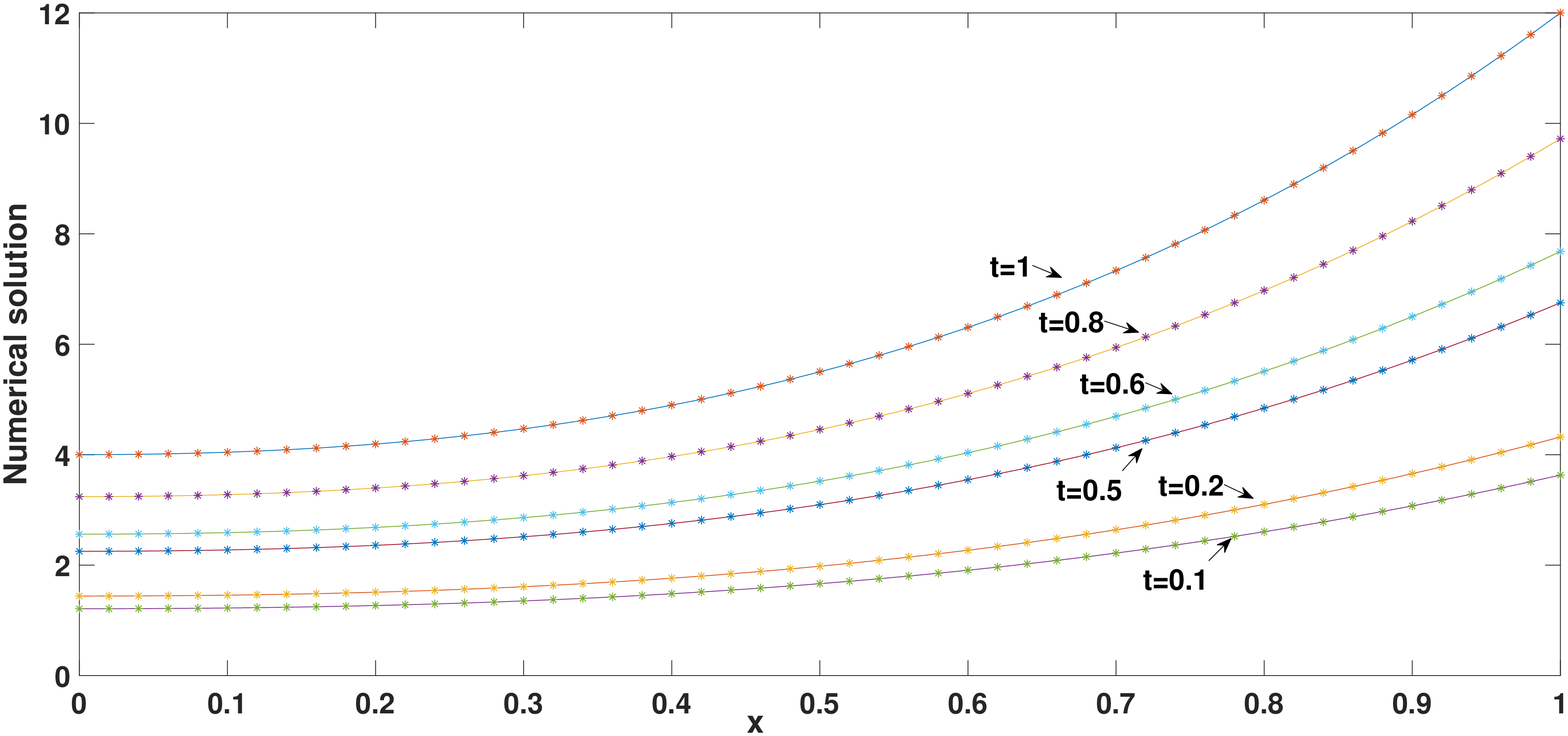}
	\caption{The numerical solution (starred line) and the exact solution (solid line) of Example \ref{eg:2} with $\mu=0.9$, $\rho=0.5$ and $N_{\varkappa}=N_{t}=50,$ at different time $t$.}\label{fig:7}
\end{figure}

\begin{example}\label{eg:3}  \emph{Let us consider the problem (\ref{eq:1})\--(\ref{eq:3}) on the domain  $(0,1) \times(0,1),$}
	\begin{equation*}
	{}_{0}^{C}D^{\mu}_{t}u(\varkappa,t)= \alpha\frac{\partial^{2}u(\varkappa,t)}{\partial \varkappa^2}+\beta\frac{\partial u(\varkappa,t)}{\partial \varkappa}-\gamma u(\varkappa,t)+\psi(\varkappa,t), 
	\end{equation*}
	\emph{with}
	\begin{equation*}
	\left\{%
	\begin{array}{ll}
	u(\varkappa,0)=\varkappa^{4}(\varkappa-1),\\	
	u(0,t)=0,\\
	u(1,t)=0,
	\end{array}%
	\right.~~~~~~~~~~~~~~~~~~~~~~~~~~~~~~~~~~~~~~~~~~~~~~~~~~~~~~~~~~~~~~~~~~~~~~~~~~~~~~~~~~
	\end{equation*} 
	\emph{and the source term}
	\begin{equation*}\psi(\varkappa,t)=\frac{6}{\Gamma(4-\mu)}t^{3-\mu}(\varkappa^{5}-\varkappa^{4})-(t^{3}+1)(4\alpha \varkappa^{2}(5\varkappa-3)+\beta \varkappa^{3}(5\varkappa-4)-\gamma \varkappa^{4}(\varkappa-1)).
	\end{equation*}	
	\emph{The exact solution of this test problem is }  $u(\varkappa,t)=(t^{3}+1)\varkappa^{4}(\varkappa-1)$. \emph{We will solve this problem with pre\--mentioned values of parameters} $r_{f}=0.02$, $\sigma=0.8$.
\end{example}
Tables \ref{5tab:9} and \ref{5tab:10} display $L_{2}$ and $L_{\infty}$ errors with the respective orders of convergence for Example \ref{eg:3} that have been evaluated for the time fractional-orders $\mu=0.5$ and $\mu=0.9$, respectively. One can see that the orders of convergence shown in Tables \ref{5tab:9} and \ref{5tab:10} are sufficiently close to $1.5$ and $1.1$ respectively. So from here, we conclude that the numerically evaluated temporal order of convergence is $O(h_{t}^{2-\mu})$ which is consistent with Theorem \ref{4:2:thm:3}. In a similar manner, the errors $L_{2}$ and $L_{\infty}$ with the respective spatial orders of convergence are tabulated in Tables \ref{5tab:11} and \ref{5tab:12} respectively.
\begin{table}
	\caption{ $L_{2}$ and $L_{\infty}$ errors and corresponding orders of convergence with $\mu=0.5$, $\rho=8.6$ and $N_{\varkappa}=1500$ for Example \ref{eg:3}.}
	\begin{tabular}{ |p{2cm} p{2cm} p{2cm} p{2cm} p{2cm} p{2cm}  p{2cm}|}		\multicolumn{7}{c}{} \\
		\hline
		$N_{t}$& $2^{3}$& $2^{4}$&$2^{5}$&$2^{6}$&$2^{7}$&$2^{8}$\\
		\hline
		$L_{2}$&  3.8050e-04& 1.4288e-04& 5.2529e-05& 1.9074e-05&6.8818e-06&2.4842e-06
		\\
	EOC&$\--$&1.4131&1.4436&1.4615&1.4708&1.4700\\
		$L_{\infty}$	& 5.7987e-04& 2.1766e-04& 7.9979e-05& 2.9009e-05&1.0437e-05& 3.7441e-06\\	
	EOC	&$\--$	&1.4137& 1.4444&1.4631& 1.4747&1.4790\\	
		\hline
	\end{tabular}\label{5tab:9}
\end{table}
\begin{table}
	\caption{ $L_{2}$ and $L_{\infty}$ errors and corresponding orders of convergence with $\mu=0.9$, $\rho=8.6$ and $N_{\varkappa}=1000$ for Example \ref{eg:3}.}
	\begin{tabular}{ |p{2cm} p{2cm} p{2cm} p{2cm} p{2cm} p{2cm}  p{2cm}|}		\multicolumn{7}{c}{} \\
		\hline
		$N_{t}$& $2^{3}$& $2^{4}$&$2^{5}$&$2^{6}$&$2^{7}$&$2^{8}$\\
		\hline
		$L_{2}$&  1.7975e-03&8.6241e-04&4.0834e-04&  1.9201e-04&8.9977e-05&  4.2101e-05\\
		EOC&$\--$ &1.0595&1.0786&1.0886&1.0936&1.0957\\
		$L_{\infty}$	& 2.7482e-03 &  1.3184e-03&6.2410e-04&  2.9340e-04&1.3742e-04& 6.4243e-05\\	
	EOC	&$\--$ &1.0597&1.0789&1.0889&1.0942&1.0970\\	
		\hline
	\end{tabular}\label{5tab:10}
\end{table}
	\begin{table}
	\caption{ {$L_{2}$ error and corresponding order of convergence for various $\mu$ with $\rho=7.4$ and $N_{t}=1000$ for Example \ref{eg:3}.}}
	\begin{tabular}{ |p{2cm} p{2cm} p{2cm} p{2cm} p{2cm} p{2cm} p{2cm}|}		\multicolumn{6}{c}{} \\
		\hline
		$\mu$	&	$N_{\varkappa}$&$2^{3}$&  $2^{4}$&$2^{5}$ & $2^{6}$& $2^{7}$ \\
		\hline
		0.2	&	$L_{2}$&3.6198e-03&9.8464e-04& 2.5136e-04& 6.3176e-05& 1.5823e-05
		\\
		&EOC&$\--$& 1.8782&1.9699& 1.9923&  1.9973 \\	
		0.4	&	$L_{2}$& 3.5812e-03& 9.7332e-04& 2.4846e-04& 6.2494e-05&   1.5704e-05
		\\
		&EOC&$\--$&1.8794  &1.9699&1.9912&1.9926\\
		
		0.6	&	$L_{2}$& 3.5451e-03& 9.6305e-04&2.4606e-04& 6.2173e-05&1.5916e-05
		\\
		&	EOC&$\--$&1.8801&1.9686&1.9847&1.9658 \\
%
		\hline
	\end{tabular}\label{5tab:11}
\end{table}
	\begin{table}
	\caption{ Maximum absolute error $L_{\infty}$ and corresponding order of convergence for various $\mu$ with $\rho=8.6$ and $N_{t}=1000$ for Example \ref{eg:3}.}
	\begin{tabular}{ |p{1.5cm} p{2cm} p{2cm} p{2cm} p{2cm} p{2cm} p{2cm}|}		\multicolumn{7}{c}{} \\
		\hline
		$\mu$	&	$N_{\varkappa}$&$2^{3}$&  $2^{4}$&$2^{5}$& $2^{6}$& $2^{7}$ \\
		\hline
		0.2	&	$L_{\infty}$	& 5.9815e-03 &   1.5523e-03&   4.0375e-04&  1.0171e-04& 2.5489e-05 \\
		&EOC	&$\--$& 1.9461&  1.9429& 1.9890&1.9965\\
		0.4	&$L_{\infty}$&   5.9791e-03& 1.5404e-03& 3.9964e-04&1.0075e-04&2.5332e-05\\
		&EOC	&$\--$& 1.9566& 1.9465& 1.9879&1.991\\
		0.6	&	$L_{\infty}$	&     5.9723e-03
		 & 1.5384e-03& 3.9623e-04& 1.0035e-04&  2.5693e-05
		 \\
		&EOC	&$\--$&  1.9569& 1.9570& 1.9813&1.9656 \\	
		\hline
	\end{tabular}\label{5tab:12}
\end{table}
	\begin{figure}
		\centering
		\includegraphics[width=1\linewidth]{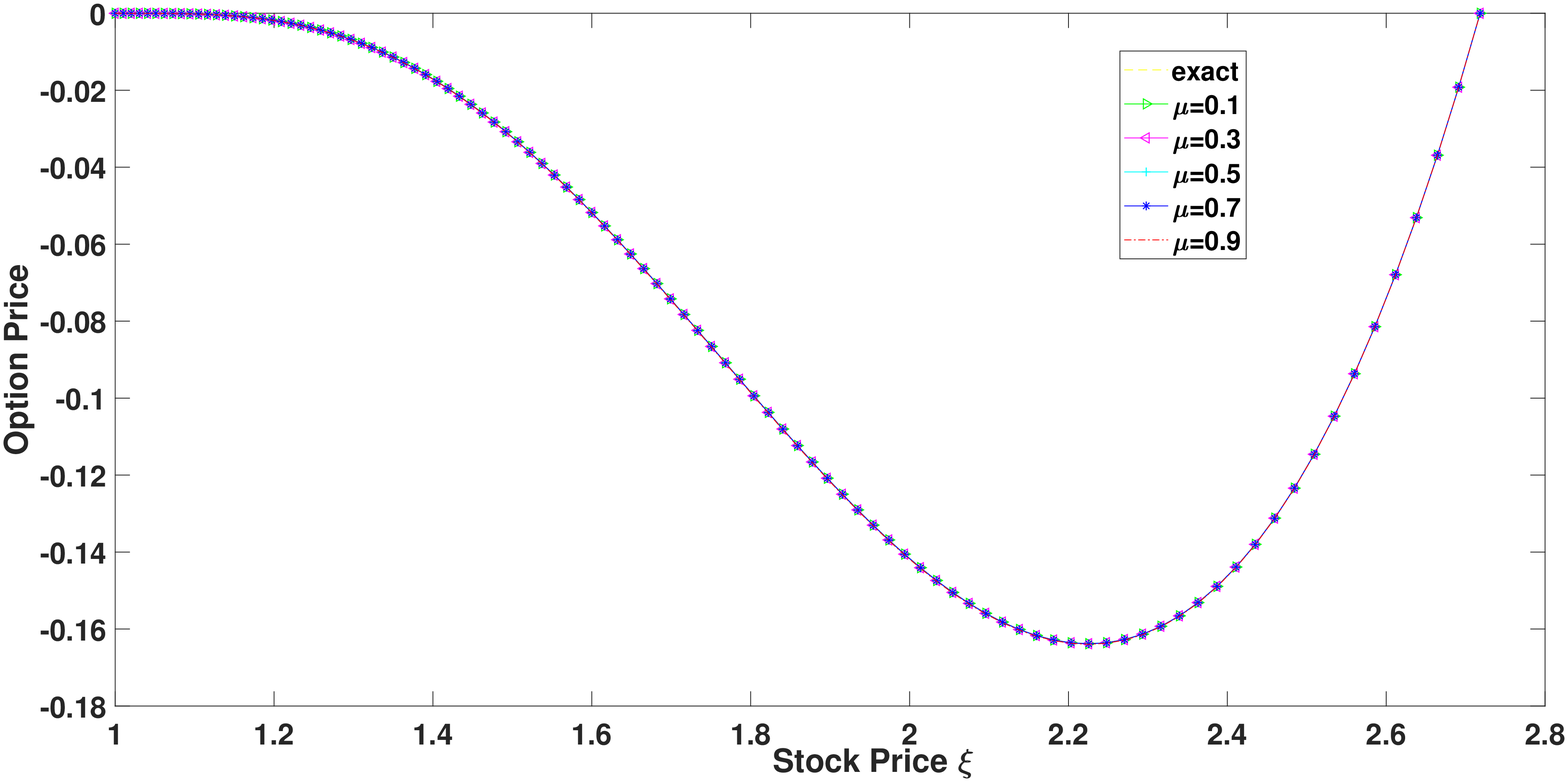}
		\caption{Th exact and numerical solutions of Example \ref{eg:3} with $\rho=8.6$ and $N_{\varkappa}=N_{t}=100$.}\label{fig:9}
	\end{figure}
	\begin{figure}
		\centering
		\includegraphics[width=1\linewidth]{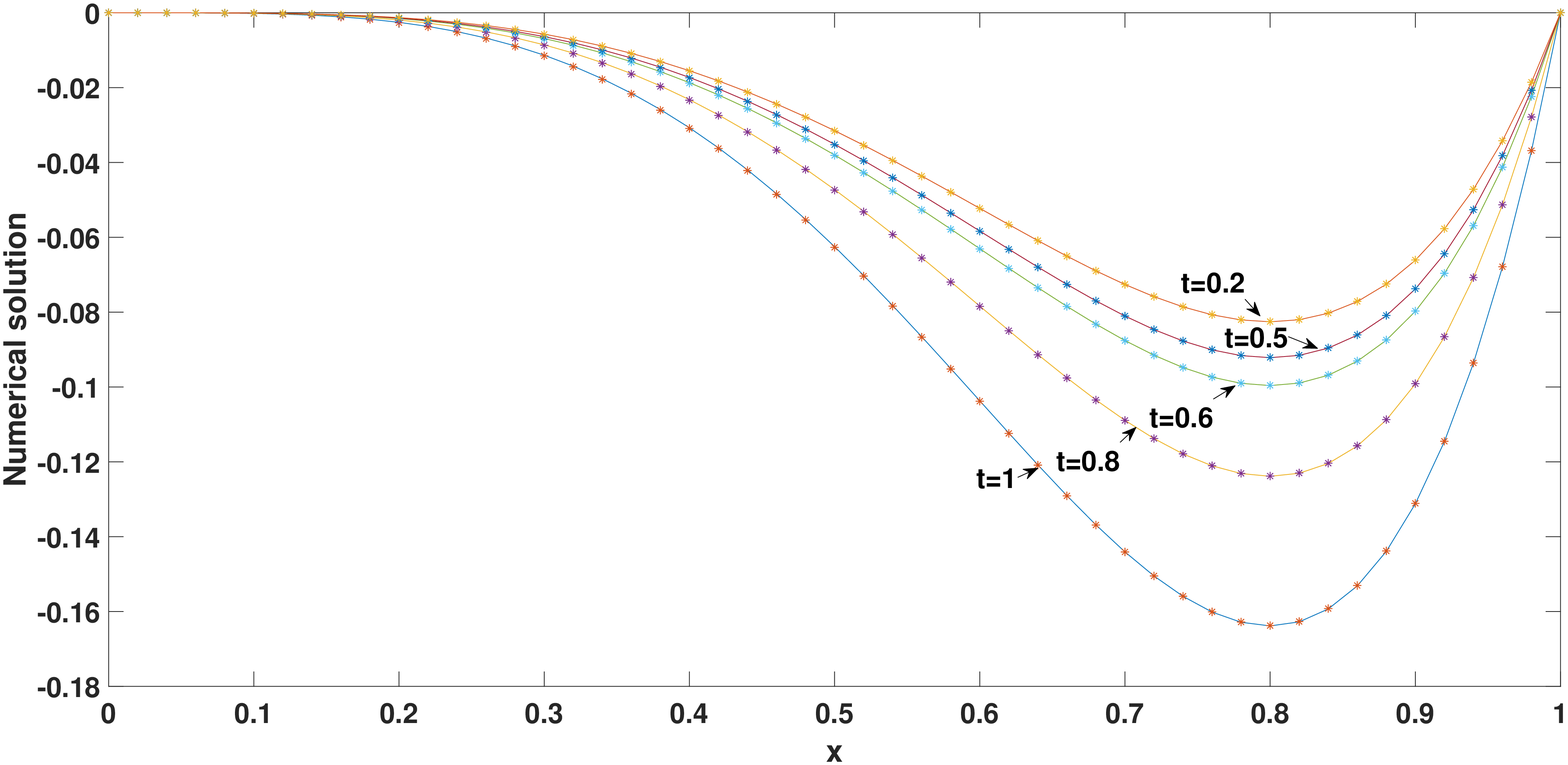}
		\caption{The numerical solution (starred line) and exact the solution (solid line) of Example \ref{eg:3} with $\mu=0.5$, $\rho=8.6$ and $N_{\varkappa}=N_{t}=50$ at different time $t$.}\label{fig:10}
	\end{figure}
	\begin{figure}
	\centering
	\begin{minipage}{.5\textwidth}
		\centering
		\includegraphics[width=1\linewidth]{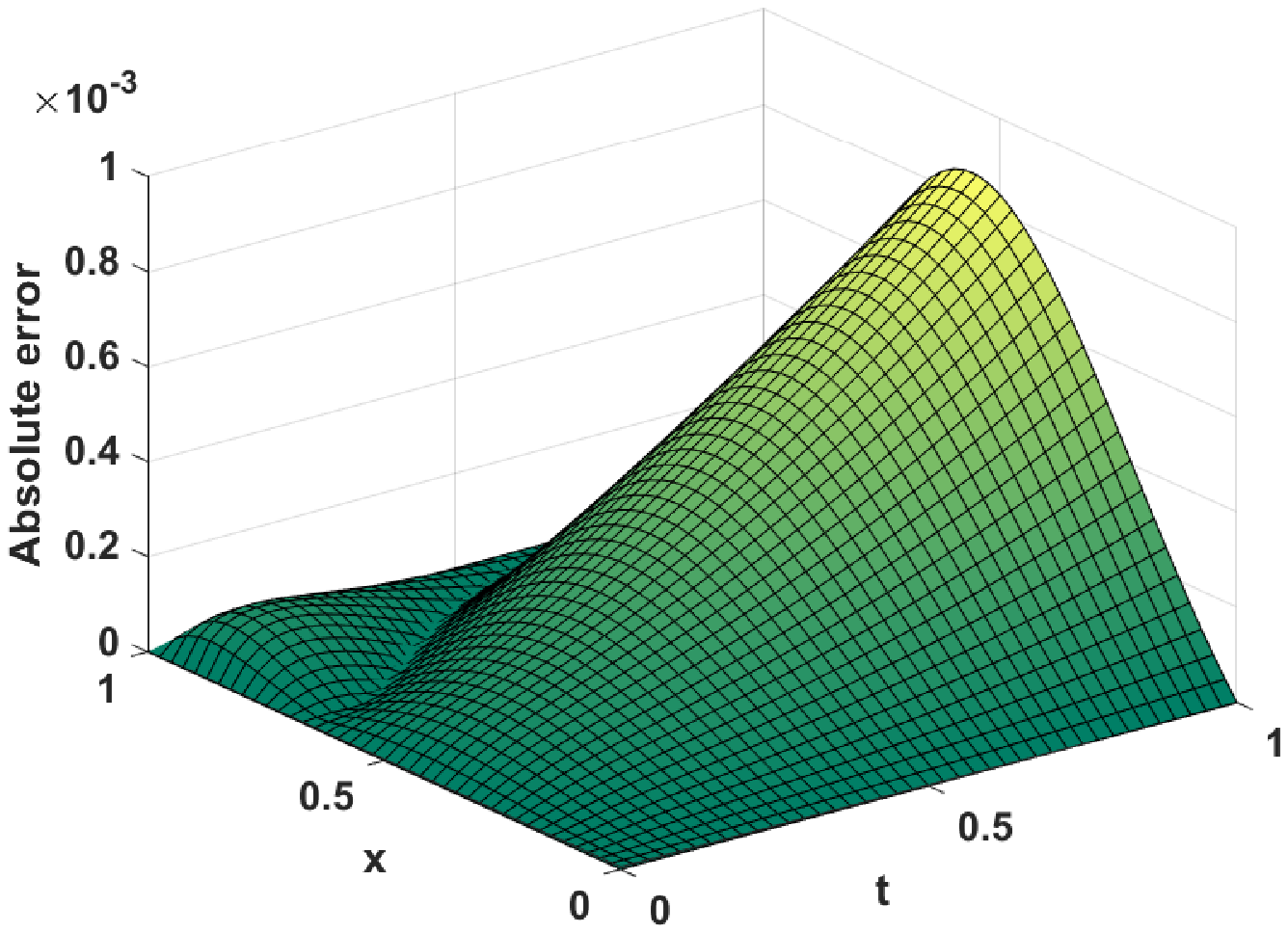}
		{$(a)\hspace{0.1cm} N_{\varkappa}=N_{t}=40$}
	\end{minipage}%
	\begin{minipage}{.5\textwidth}
		\centering
		\includegraphics[width=1\linewidth]{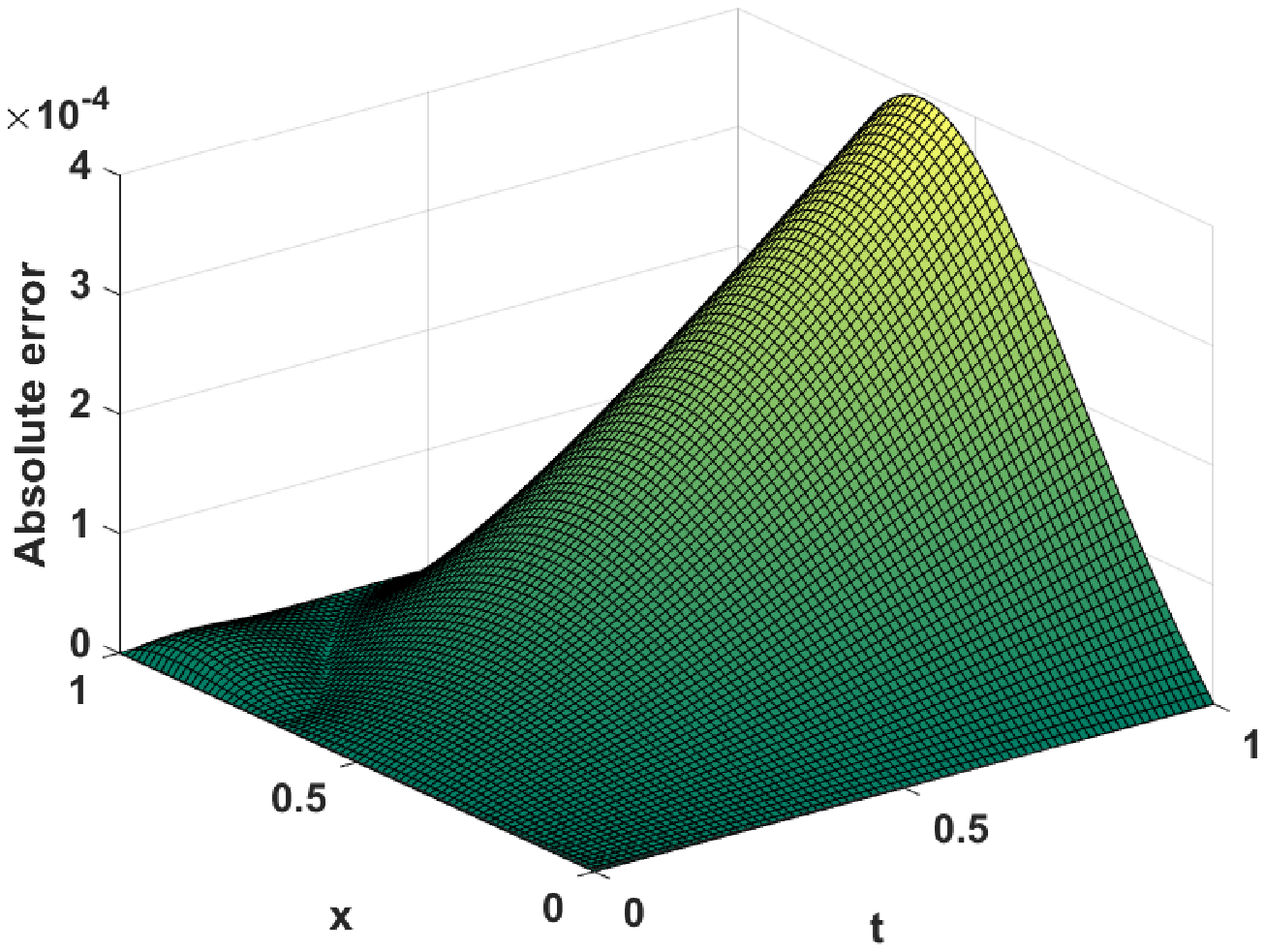}
		{$(b)\hspace{0.1cm} N_{\varkappa}=N_{t}=80$}
	\end{minipage}
\end{figure}
\begin{figure}
	\centering
	\begin{minipage}{.5\textwidth}
		\centering
		\includegraphics[width=1\linewidth]{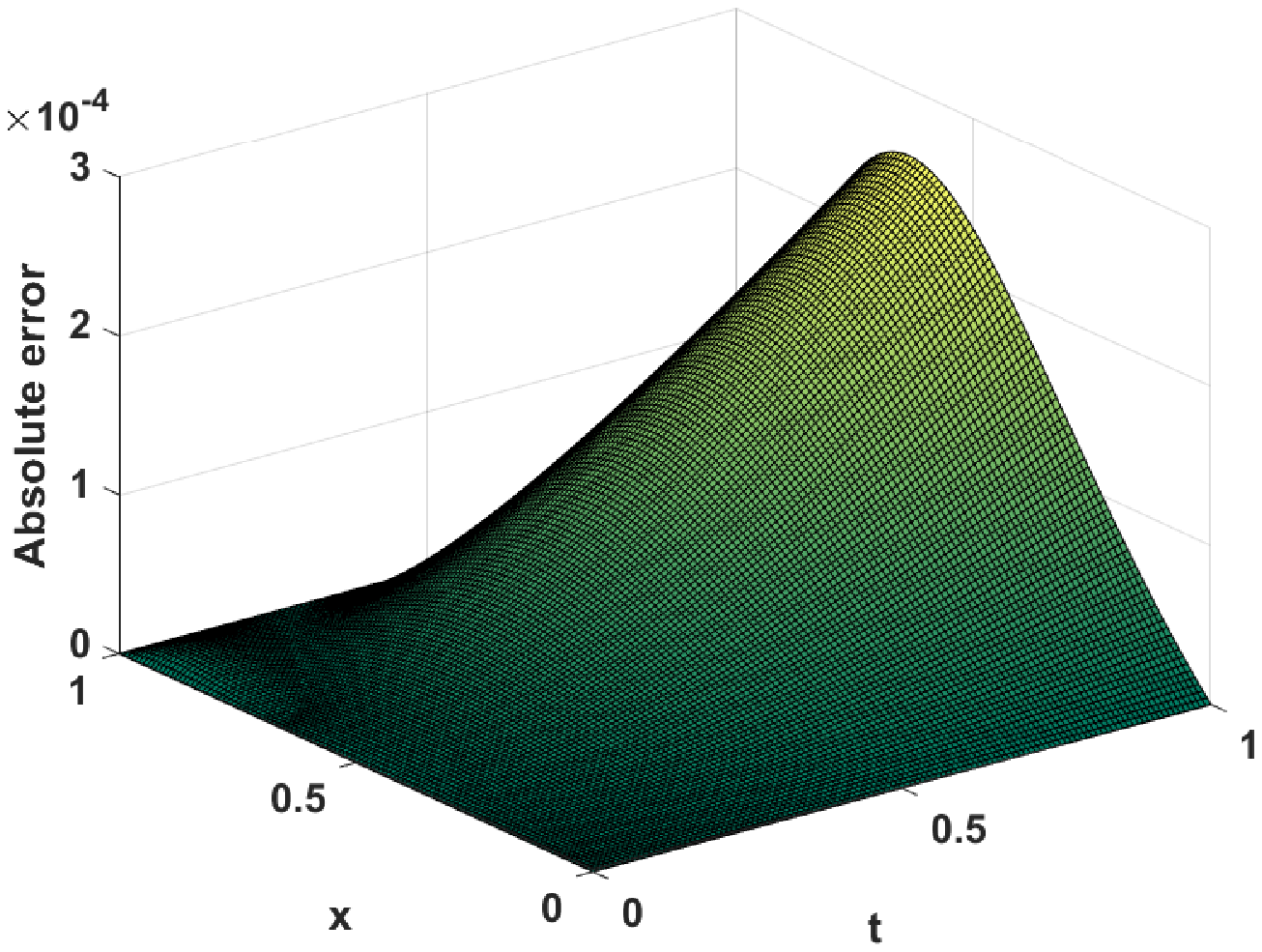}
		{$(c)\hspace{0.1cm} N_{\varkappa}=N_{t}=120$}
	\end{minipage}%
	\begin{minipage}{.5\textwidth}
		\centering
		\includegraphics[width=1\linewidth]{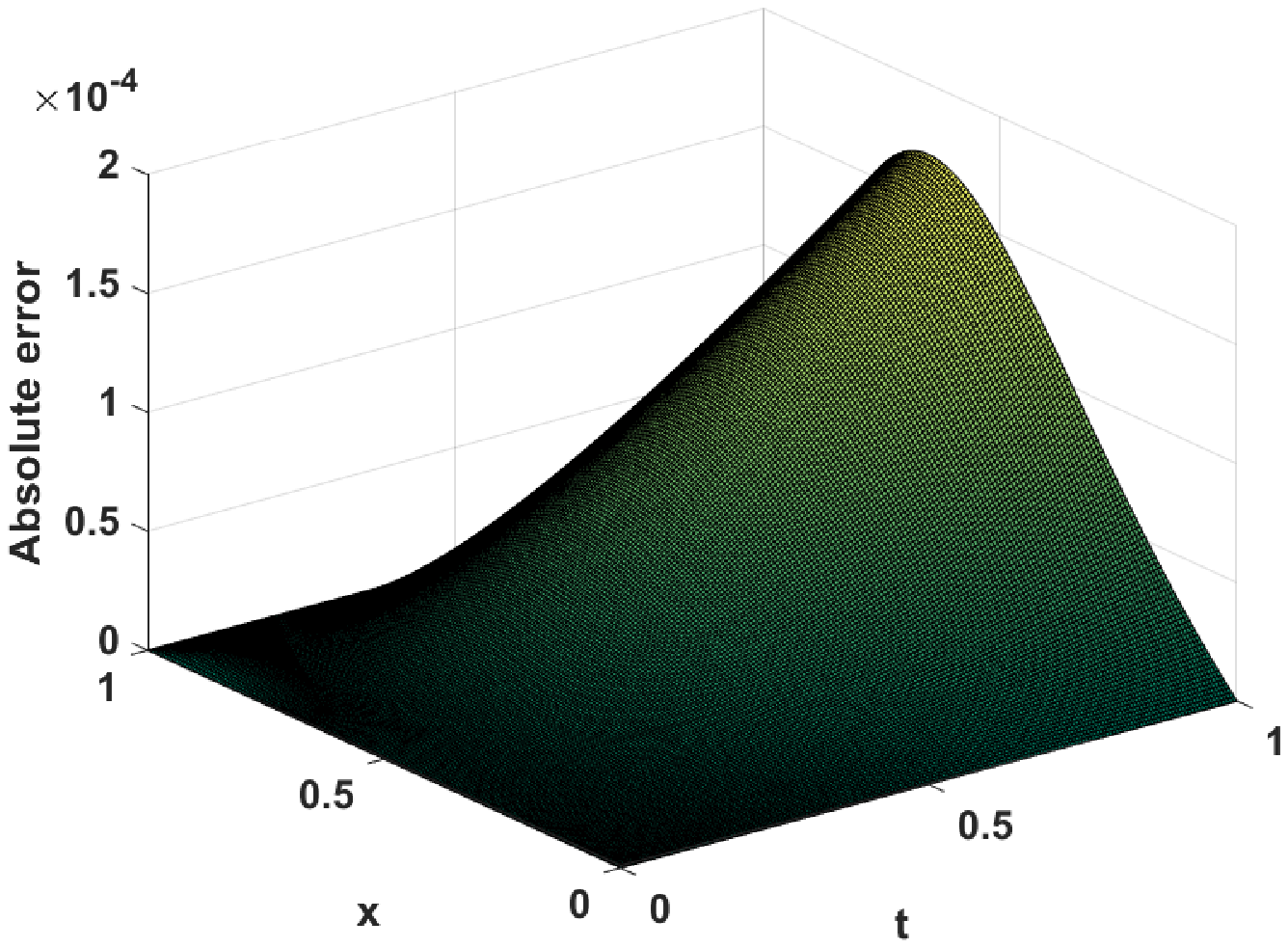}
		{$(d)\hspace{0.1cm} N_{\varkappa}=N_{t}=160$}
	\end{minipage}
	\caption{Maximum absolute error plots for different mesh sizes for Example \ref{eg:3} with $\mu=0.99$ and $\rho=8.6.$}\label{fig:11}
\end{figure}
	\begin{example}\label{eg:4}  \emph{Let us consider the TFBSM \cite{hang2016numerical1} } 
		\begin{equation}
	\frac{\partial^{\mu} \mathscr{V}(\xi,\tau)}{\partial \tau^{\mu}} +\frac{\sigma^{2}\xi^{2}}{2}\frac{\partial^{2} \mathscr{V}(\xi,\tau)}{\partial \xi^2}+(r_{f}-D_{Y})\xi\frac{\partial  \mathscr{V}(\xi,\tau)}{\partial \xi}-r_{f}  \mathscr{V}(\xi,\tau)=0, \hspace{0.3cm} (\xi,\tau)\in (\xi_{I_{p}},\xi_{F_{p}}) \times (0,\widetilde{T}),
	\end{equation}
	\begin{equation*}
	\left\{%
	\begin{array}{ll}
\mathscr{V}(\xi,\widetilde{T})=\phi(\xi),\\	
	\mathscr{V}( \xi_{I_{p}},\tau)=\mathscr{H}(\tau),\\
	\mathscr{V}( \xi_{F_{p}},\tau)=\mathscr{G}(\tau),
	\end{array}%
	\right.~~~~~~~~~~~~~~~~~~~~~~~~~~~~~~~~~~~~~~~~~~~~~~~~~~~~~~~~~~~~~~~~~~~~~~~~~~~~~~~~~~
	\end{equation*}
\end{example}
The above model describes different option values depending on the functions $\phi(\xi)$, $\mathscr{H}(\tau)$, and $\mathscr{G}(\tau)$.
\begin{itemize}
	\item {If
		$\phi(\xi) = \max\{\widetilde{K}-\xi,0\}$, $\mathscr{H}(\tau) = \widetilde{K}e^{-r_{f}(\widetilde{T}-\tau)}$ and $\mathscr{G}(\tau) =D_{Y}=0$, then the model represents the European put option.}
	\item {If $\phi(\xi) = \max\{\xi-\widetilde{K} , 0\}$, $\mathscr{H}(\tau) = D_{Y} = 0$ and $\mathscr{G}(\tau) =\xi_{I_{f}}-\widetilde{K}e^{-r_{f}(\widetilde{T}-\tau)},$ then the model represents the European call option.}
	\item { And if $\phi(\xi) = \max\{\xi-\widetilde{K} , 0\}$ and $\mathscr{H}(\tau)=\mathscr{G}(\tau)= 0$, then the model describes the European double barrier knock-out call option.}
\end{itemize}
 For solving European call and European put option models numerically we have taken the parameters $\sigma = 0.55$, $r_{f} = 0.05$,
$\xi_{I_{p}}= 0.1(I_{p} = -2.3)$, $\xi_{F_{p}} = 100(F_{p} = 4.6)$, $\widetilde{T} = 1$(year), and the strike price $\widetilde{K} = 50$. Also, we take the parameters $\sigma=0.55$, $r_{f} = 0.03$, $\xi_{I_{p}}= 3(I_{p} = 1.1)$, $\xi_{F_{p}}= 15(F_{p} = 2.7)$, $\widetilde{T} = 1$(year), the dividend yield $D_{Y} = 0.01$, and the strike price $\widetilde{K}= 10,$ to study the European double barrier knock-out call option model numerically.

Figures \ref{fig:12} and \ref{fig:13} show how the orders of fractional derivative affect the European call option and put option, respectively. From these two figures, we can observe that when the stock price $\xi$ is less or greater than the strike price $\widetilde{K}$, the option price is slightly affected by the order of the time-fractional derivative. And when the stock price $\xi$ is close to the strike price $\widetilde{K}$, the option price is significantly affected by the time-fractional derivative order. Further, Figure \ref{fig:14} show how the different parameters affect the European put option price governed by TFBSM. From Figure \ref{fig:14}$(a)$, it can be observed that the rate of interest and the options are inversely proportional, i.e the higher the rate of interest, the lower the option. Similarly, Figure \ref{fig:14}$(b)$ reflects, how the stock price volatility influences the option price and supports a well-known statement \textquotedblleft The higher the risk, the higher the return".
In Figure \ref{fig:15}, the graph is plotted between the stock price $\xi$ and the double barrier option price $\mathscr{V}$ for various values of time-fractional order $\mu$. It is worth noting here that for $\mu=1$, the TFBSM converts to the classical Black-Scholes model. This Figure \ref{fig:15} depicts that the double barrier option price is highly influenced by the time-fractional order. More specifically, the option price is inversely proportional to the fractional-order when the stock price is greater than or near the strike price $\widetilde{K}$. And the peak of the option price curve occurred corresponding to $\mu=0.2$. This tells us that TFBSM can explain the jump movement of the problem much more clearly than the classic Black\--Scholes model.
\begin{figure}
	\centering
	\includegraphics[width=1\linewidth]{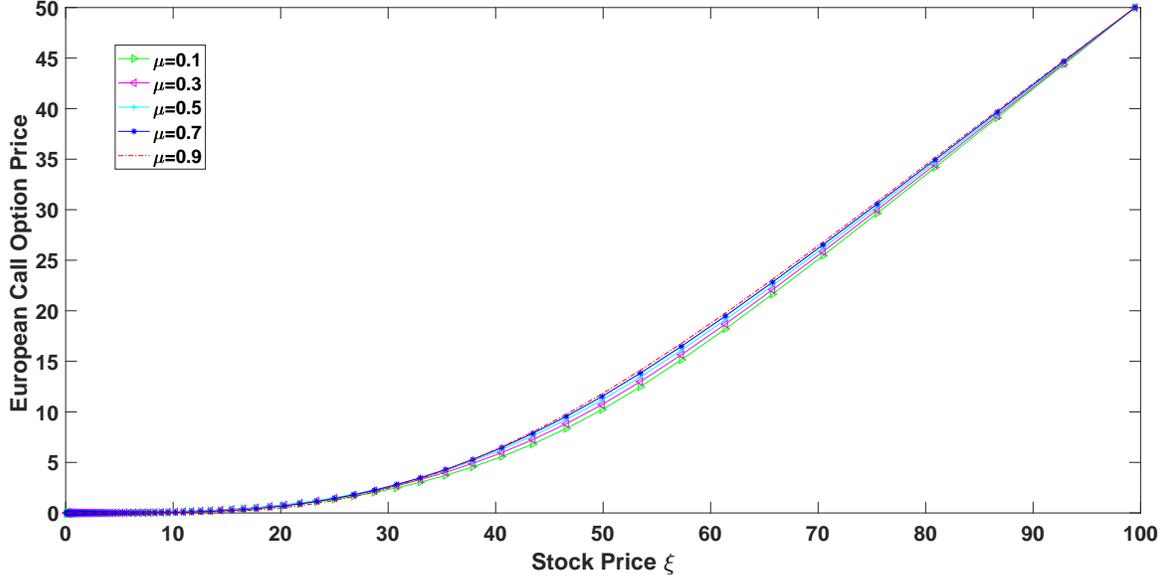}
	\caption{European call option curves with $\rho=1.5$, $N_{\varkappa}=N_{t}=100$ and different $\mu$  for Example \ref{eg:4}.}\label{fig:12}
\end{figure}
\begin{figure}
	\centering
	\includegraphics[width=1\linewidth]{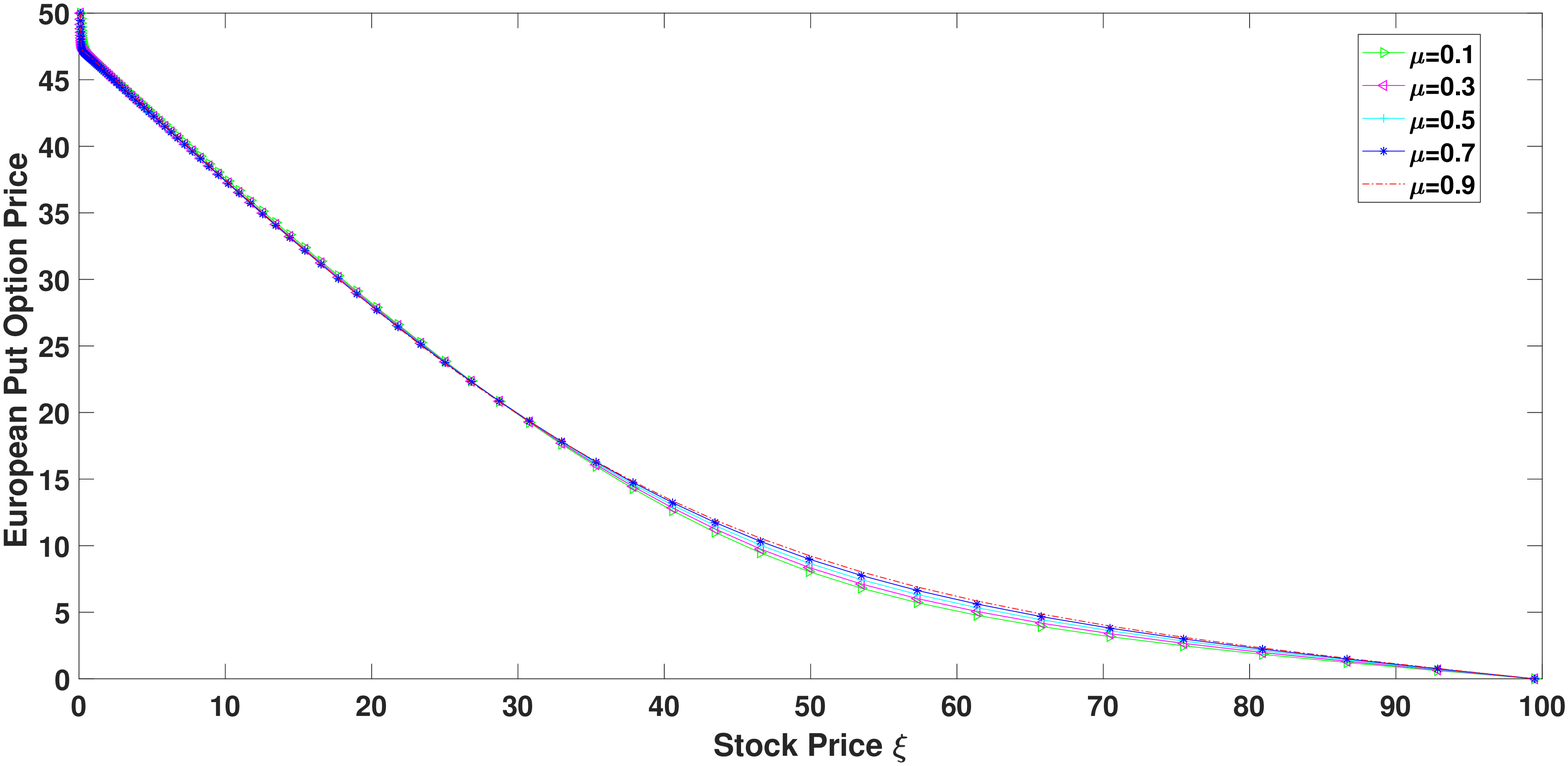}
	\caption{European put option curves with $\rho=1.5$, $N_{\varkappa}=N_{t}=100$ and different $\mu$  for Example \ref{eg:4}.}\label{fig:13}
\end{figure}
\begin{figure}
	\centering
	\begin{minipage}{.5\textwidth}
		\centering
		\includegraphics[width=1\linewidth]{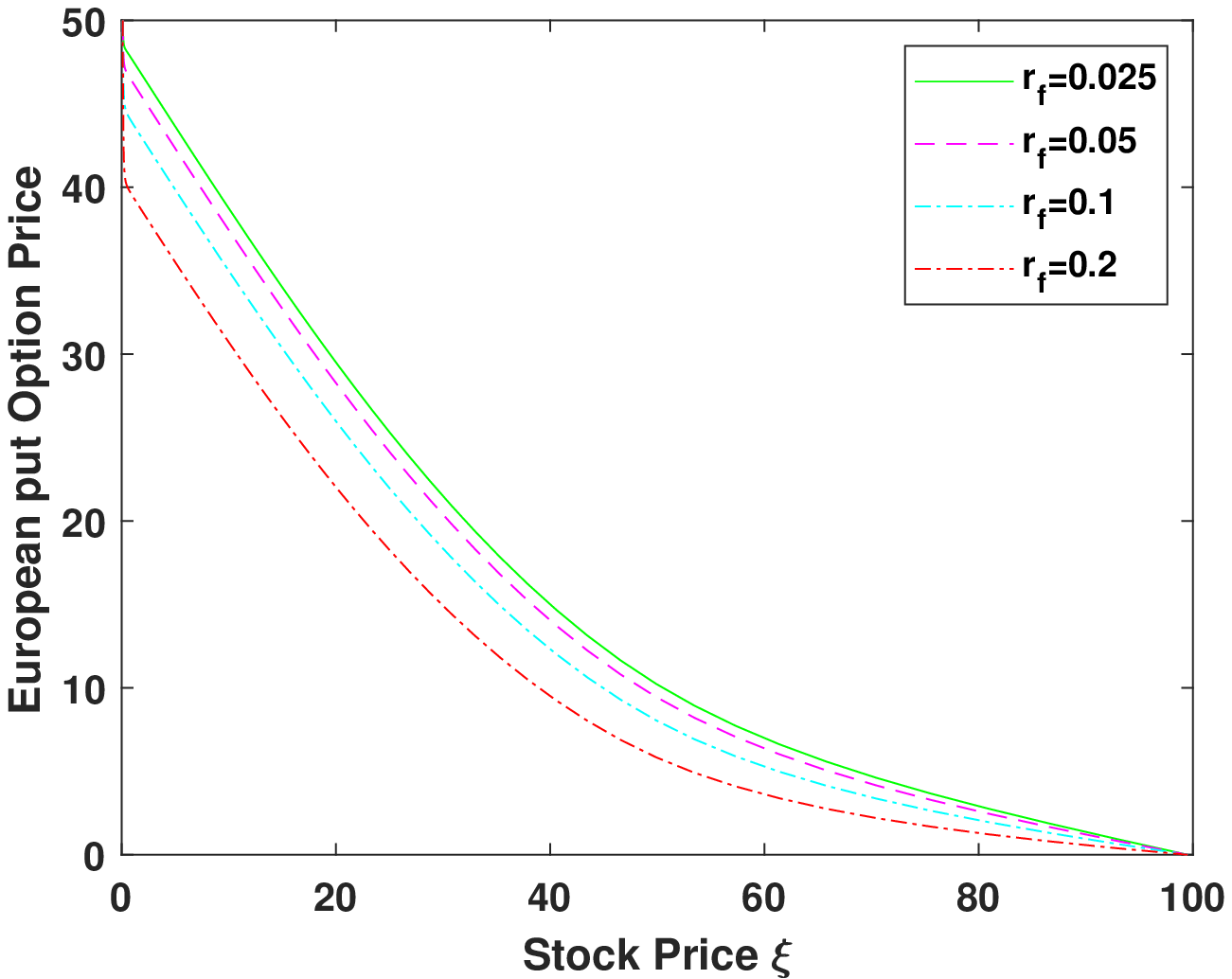}
		{$(a)\hspace{0.1cm} \sigma=0.6, \widetilde{K}=50, T=1$}
	\end{minipage}%
	\begin{minipage}{.5\textwidth}
		\centering
		\includegraphics[width=1\linewidth]{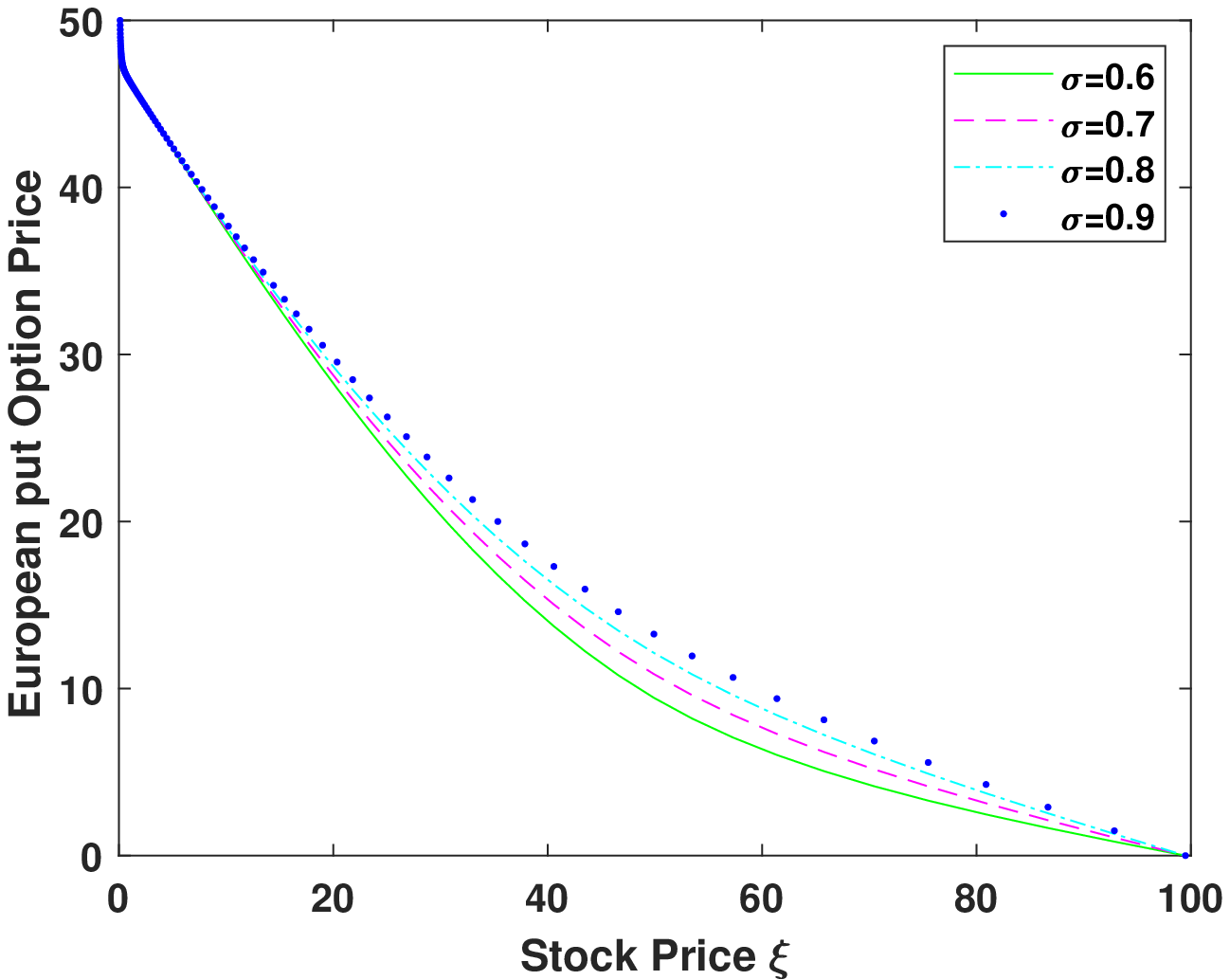}
		{$(b)\hspace{0.1cm} r_{f}=0.05, \widetilde{K}=50, T=1$}
	\end{minipage}
	\caption{European put option curves with different values of parameters for Example \ref{eg:4}.}\label{fig:14}
\end{figure}
\begin{figure}
	\centering
	\includegraphics[width=1\linewidth]{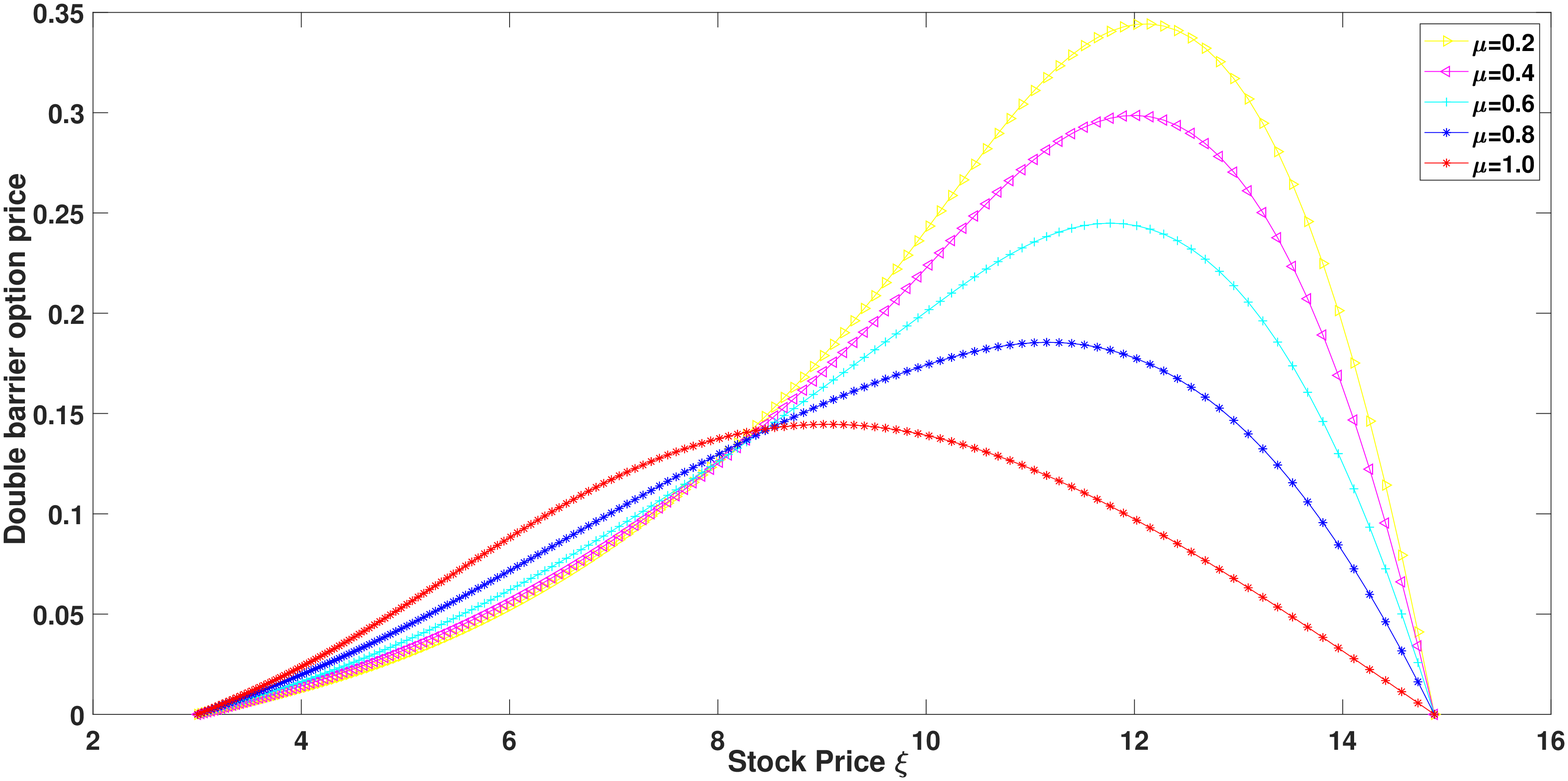}
	\caption{Double barrier option curves with $\rho=1.5$, $N_{\varkappa}=N_{t}=150$ and different $\mu$  for Example \ref{eg:4}.}\label{fig:15}
\end{figure}

\section{Conclusion}\label{sec:6}
\noindent In this paper, an efficient collocation method based on exponential B-spline functions is introduced to solve the TFBSM governing European options. First, we have changed the modified R-L fractional derivative operator to the Caputo fractional derivative operator by applying the variable transformation, then used the exponential B\--spline functions to discretize the space derivative and a finite difference method to discretize the Caputo fractional derivative. As a result of the use of the exponential B-spline collocation method, a tri-diagonal algebraic system has been obtained that can be solved by the Thomas algorithm. Furthermore, the proposed numerical scheme has been shown to be unconditionally stable via the von-Neumann method. The method is implemented on a number of numerical examples. And the obtained results confirm that the method is capable of approximating the TFBSM. In addition, as an application, the numerical scheme proposed for the TFBSM has been used to price several different European options and it has been observed that the order of the time-fractional derivative has a great impact on the option prices. Since the proposed scheme works well for the TFBSM, so it is our intention to extend this idea to solve other fractional problems numerically.

\section*{Declarations}
\noindent \textbf{Conflict of interest:} The authors declare that they have no known competing financial interests or personal relationships that could have appeared to influence the work reported in this paper.

\bibliographystyle{elsarticle-num}
\bibliography{references5}

\end{document}